\documentclass[a4paper]{amsart}
\usepackage{amssymb,color,enumerate}
\usepackage{latexsym,url}
\usepackage{graphicx}
\usepackage{pdfsync,url}
\usepackage{amsmath}
\usepackage{amsfonts}
\usepackage{wrapfig}
\usepackage{amsfonts, ulem}
\allowdisplaybreaks

\usepackage[square,comma,sort&compress]{natbib}
\setcitestyle{numbers}
\usepackage{mathrsfs} 

\theoremstyle{plain}
 
 \newtheorem{prop}{Proposition}[section]

\theoremstyle{definition}
 \newtheorem{exm}{Example}[section]
 \newtheorem{rem}{Remark}[section]
 \newtheorem{dfn}{Definition}[section]
\numberwithin{equation}{section}
\renewcommand{\le}{\leqslant}\renewcommand{\leq}{\leqslant}
\renewcommand{\ge}{\geqslant}\renewcommand{\geq}{\geqslant}
\renewcommand{\setminus}{\smallsetminus}

\newcommand{\bfE}{\mathbf{E}}

\newcommand{\h}{h(t)}

\newcommand{\R}{\mathbb{R}}

\newcommand{\M}{\mathbb{M}}

\newcommand{\Prob}{\mathbf{P}}
\newcommand{\E}{\mathbb{E}}

\newcommand{\vague}{\stackrel{\lower0.2ex\hbox{$\scriptscriptstyle
                    \it{v} $}}{\rightarrow}}
\newcommand{\weak}{\stackrel{\lower0.2ex\hbox{$\scriptscriptstyle
                    \it{w} $}}{\rightarrow}}
\newcommand{\what}{\stackrel{\lower0.2ex\hbox{$\scriptscriptstyle
                    \it{\hat{w}} $}}{\rightarrow}}
\newcommand{\eqdis}{\stackrel{\lower0.2ex\hbox{$\scriptscriptstyle
                    \mathrm{d}$}}{=}}
\newcommand{\distr}{\stackrel{\lower0.2ex\hbox{$\scriptscriptstyle
                    \it{d} $}}{\rightarrow}}

\def\bX{\boldsymbol X}
\def\bY{\boldsymbol Y}

\def\bV{\boldsymbol V}
\def\bZ{\boldsymbol Z}

\def\bC{\mathbb{C}}

\def\bone{\boldsymbol 1}

\def\bzero{\boldsymbol 0}

\def\bx{\boldsymbol x}
\def\bz{\boldsymbol z}

\def\by{\boldsymbol y}
\def\binfty{\boldsymbol \infty}
\def\bTheta{\boldsymbol \Theta}

\def\bbeta{\boldsymbol \beta}

\def\gpolar{\text{GPOLAR}}

\def\axes{\text{axes}}

\newcommand\independent{\protect\mathpalette{\protect\independenT}{\perp}}
\def\independenT#1#2{\mathrel{\rlap{$#1#2$}\mkern2mu{#1#2}}}

\allowdisplaybreaks 
\numberwithin{equation}{section}
\def\Hillish{\text{Hillish}}
\def\CEV{\text{CEV}}
\def\MRV{\text{MRV}}
\def\HRV{\text{HRV}}
\newcommand{\cinP}{\stackrel{\lower0.2ex\hbox{$\scriptscriptstyle
                    \it{P} $}}{\rightarrow}}
\def\Pick{\text{Pickandsish}}
\def\bbeta{\mathbb{\eta}}
\def\bbxi{\mathbb{\xi}}
\newcommand\inv{^\leftarrow}

\title[Generating MRV and HRV]{{Models with Hidden Regular Variation:\\ generation and detection}}



\author[Das]{Bikramjit Das}
\address{
Bikramjit Das\\ESD  \\ 
Singapore University of Technology and Design,
Singapore}
\email{bikram@sutd.edu.sg}

\author[Resnick ]{Sidney I. Resnick}
\address{Sidney I. Resnick\\School of ORIE, Cornell University,
Ithaca, NY 14853 USA} \email{sir1@cornell.edu}

\thanks{B. Das was supported by SRG-ESD-2012-047. S. Resnick was supported by Army MURI grant
  W911NF-12-1-0385 to Cornell University. Resnick acknowledges
  hospitality, space and support from SUTD during a visit January 2014.} 

\begin{document}

\setcounter{page}{1}
\thispagestyle{empty}

\begin{abstract}
We review definitions of multivariate regular variation (MRV) and hidden
regular variation (HRV) for distributions of random vectors and then
summarize methods for generating models exhibiting both properties. We
also discuss diagnostic techniques that detect these properties in
multivariate data and indicate when  models exhibiting both  MRV and
HRV are plausible fits for the data. We illustrate our techniques on
simulated data and also two real Internet data sets.
\end{abstract}
\bibliographystyle{plainnat}
\maketitle

\section{Introduction}
{Data exhibiting heavy tails appear naturally in many contexts, for
  example hydrology \cite{anderson:meerschaert:1998}, finance
  \cite{smith:2003}, insurance
  \cite{embrechts:kluppelberg:mikosch:1997}, Internet traffic and
  telecommunication \cite{crovella:bestavros:taqqu:1998} and risk
  assessment \cite{das:embrechts:fasen:2013,ibragimov:jaffee:walden:2011}.
  Often the
observed  data  are multi-dimensional  with heavy tailed marginal
distributions and come from complex systems and we must study the
dependence structure among the
components. 

The study of multivariate heavy-tailed models is facilitated by
the ability to  generate such models. Moreover, a
  generation technique helps in stress-testing  worst-case
  scenarios. In the first part of this paper we consider
  several   generation techniques and discuss their strengths and
  weaknesses. 

A second theme of this paper is the development of diagnostics for
detecting and  identifying multivariate heavy tailed models prior to 
estimating model parameters. The second part of this
  paper deals with this.} 

{
\subsection{Outline}\label{subsec:outline}
The mathematical framework for the study of multivariate heavy tails
is regular variation of measures.
We provide a careful review of the
definitions of multivariate regular variation (MRV) and hidden regular
variation (HRV) in Section \ref{subsec:regVarMod} and list the
notations we use in Section \ref{subsec:notation}.  In Section
\ref{sec:genMRV} we discuss methods for generating regularly varying
models on $\E=[0,\infty)^2\setminus \{(0,0)\}$ and $\E_0=(0,\infty)^2$
when  the asymptotic limit measures are specified. The described methods  are
relatively easy to implement.

 In Section \ref{sec:genMRVHRV} we
discuss how to create models that 
exhibit both MRV and HRV. Both MRV and HRV are asymptotic
models with curious properties which are often ignored or
misinterpreted when attempting to generate finite samples exhibiting
such properties.  
We review three model generation methods that yield the
asymptotic properties of both MRV on $\E$ and HRV  on $\E_0$
and discuss characteristics of each method. These methods}
are called  (i) the mixture 
method, (ii) the multiplication method and (iii) the additive method.
We give particular attention to the recently
proposed additive generation method of \cite{weller:cooley:2013} and
show that there are identifiability issues in the sense that asymptotic
parameters may not be coming from the anticipated summand of the representation.
Accompanying simulation examples illustrate our discussion. 

 Section \ref{sec:detect} gives techniques for
detecting when data is consistent with a model exhibiting MRV and
HRV. These techniques rely on the fact that under 
broad conditions, if  a vector $\bX$ has a  multivariate
 regularly varying distribution on a cone $\mathbb{C}$, then under a
{\it generalized polar coordinate transformation\/} (see
\eqref{eq:defgpolar}), the transformed 
vector satisfies a conditional extreme value (CEV) model for which 
detection techniques exist from \cite{das:resnick:2011b}.
This methodology goes beyond 
  one dimensional techniques such as checking one
  dimensional marginal distributions are heavy tailed
  or checking one dimensional functions of the data vector such as maximum and minimum
component are heavy tailed.  

In Section
  \ref{sec:detect:data}, we give two examples of our detection and
  model estimation techniques applied to
Internet downloads and HTTP response data.

\subsection{Regularly varying distributions on cones.}\label{subsec:regVarMod}
We review material from \cite{hult:lindskog:2006a, das:mitra:resnick:2013,
  lindskog:resnick:roy:2013} describing the framework for the
definition of MRV and HRV and then specialize to two dimensions.

Let $\mathbb{X}$ be a metric space with metric $d(\bx,\by)$ satisfying
\begin{equation}\label{eq:Metricscale}
d(c\bx, c\by)=cd(\bx,\by), \qquad  c>0,  (\bx,\by) \in \mathbb{X} \times \mathbb{X}.
\end{equation}
If $d(\cdot,\cdot)$ is defined by a norm, \eqref{eq:Metricscale} is
satisfied. Hence in finite dimensional Euclidean space,
\eqref{eq:Metricscale} can always be satisfied.
A flexible framework for discussing regular variation is measure
convergence {defined by} $\M$-convergence
\cite{lindskog:resnick:roy:2013, das:mitra:resnick:2013}) on a closed
cone $\mathbb{C}\subset \mathbb{X}$ with a closed cone 
$\mathbb{C}_0 \subset  \mathbb{C}$ deleted. The {concept} of a cone
requires specifying a definition of scalar multiplication $(c,\bx) \mapsto c\bx$
from $\mathbb{R}_+\times \mathbb{X} \mapsto \mathbb{X}$. 
In this paper, the metric space is
 Euclidean and scalar multiplication is the usual one.
A cone $\mathbb{C} $ is closed under
scalar multiplication: If $\bx \in
\mathbb{C}$ then $c\bx \in \mathbb{C}$ for $c>0$.
A subset $\Lambda \subset \mathbb{C} \setminus \mathbb{C}_0$ is {\it
  bounded away from\/} $\mathbb{C}_0$ if $d(\Lambda,\mathbb{C}_0)>0$.
 The two cases of most interest
are
\begin{enumerate}
\item $\mathbb{C} = \mathbb{R}_+^2$
 and $\mathbb{C}_0 =\{\bzero\}$. Then $\E:=\mathbb{C} \setminus
 \mathbb{C}_0
= \mathbb{R}_+^2 \setminus \{\bzero\}$ is the space for defining
$\M$-convergence appropriate for regular variation of distributions of
positive random vectors.

\item $\mathbb{C} = \mathbb{R}_+^2$
 and $\mathbb{C}_0 =\{\bx: \wedge_{i=1}^2  x_i =0\}:=[\text{axes}]$. Then $\E_0:=\mathbb{C} \setminus
 \mathbb{C}_0$, the first quadrant without its axes, is the space for
 defining $\M$-convergence appropriate for HRV. 
\end{enumerate}

A random vector $\bZ\geq \bzero$ is regularly varying on $\mathbb{C}
\setminus \mathbb{C}_0$ if there exists a regularly varying function
$b(t) \in RV_{1/\alpha}$, $\alpha >0$ called the {\it scaling
  function\/} and a measure $\nu(\cdot) 
\in \M(\mathbb{C} 
\setminus \mathbb{C}_0)$ called the {\it limit or tail measure\/}
{such that} as $t \to\infty$,
\begin{equation}\label{eq:RegVarMeas}
t\Prob[ \bZ/b(t) \in \cdot \,] \to \nu(\cdot),
\end{equation}
in $\M(\mathbb{C}
\setminus \mathbb{C}_0)$, the set of measures on $\mathbb{C}
\setminus \mathbb{C}_0$ which are finite on sets bounded away from
$\mathbb{C}_0$ \citep{das:mitra:resnick:2013,
  hult:lindskog:2006a,lindskog:resnick:roy:2013}. 
We write $\bZ \in MRV(\alpha, b(t), \nu, \mathbb{C}
\setminus \mathbb{C}_0)$.
Since $b(t) \in RV_{1/\alpha}$, 
$\nu(\cdot) $ has a scaling property 
\begin{equation}\label{eq:limMeasScales}
\nu(c \cdot) =c^{-\alpha} \nu(\cdot),\qquad c>0.
\end{equation}
When $\mathbb{C}=\mathbb{R}_+^2$, $\mathbb{C}_0=\{\bzero\}$ and
$\nu$ satisfies $\nu(\bx,\binfty)=0$ for all $\bx>\bzero$ so that
$\nu$ concentrates on the axes, we say $\bZ$ possesses {\it asymptotic
  independence\/} \cite{resnickbook:2007, dehaan:ferreira:2006, resnickbook:2008}.
It is convenient to translate \eqref{eq:RegVarMeas} and
\eqref{eq:limMeasScales} using generalized polar coordinates
\citep{lindskog:resnick:roy:2013, das:mitra:resnick:2013}. Set 
$\aleph_{\mathbb{C}_0}=\{\bx \in \mathbb{C} \setminus
\mathbb{C}_0: d(\bx, \mathbb{C}_0)=1\},$
 the locus of points at distance 1 from the deleted region
$\mathbb{C}_0$. Define $\gpolar: \mathbb{C}\setminus \mathbb{C}_0 \mapsto (0,\infty)\times
\aleph_{\mathbb{C}_0}$ by
\begin{equation}\label{eq:defgpolar}
\gpolar (\bx) =\Bigl(d(\bx,\mathbb{C}_0) , \frac{\bx}
{d(\bx,\mathbb{C}_0)}\Bigr)
\end{equation}
Then (\cite{das:mitra:resnick:2013, lindskog:resnick:roy:2013})
\eqref{eq:RegVarMeas} and
\eqref{eq:limMeasScales} 
are equivalent to  
\begin{equation}\label{eq:limMeasPolar}
t\Prob[\gpolar (\bZ)/b(t) \in \cdot \,] \to \nu_\alpha  \times
S(\cdot)  { = \nu\circ\gpolar^{-1}},
\end{equation}
in $\M\bigl((0,\infty) \times \aleph_{\mathbb{C}_0} \bigr)$
where $\nu_\alpha (x, \infty) =x^{-\alpha}, \,x>0,\,\alpha>0$
and $S(\cdot) $ is a probability measure on $\aleph_{\mathbb{C}_0}
$. {One should note that the transformation $\gpolar$ depends on the
  cone $\bC_0$; this dependence should be understood from the context.} 

We focus on regular variation for $p=2$ and the two choices of
$\mathbb{C}$ and $\mathbb{C}_0$ which yield the spaces
\begin{enumerate}
\item $\E : =\mathbb{R}_+^2 \setminus \{\bzero\}$.
\item $\E_0=\mathbb{R}_+^2 \setminus \{\bx : x_1\wedge
  x_2=0\}=:\mathbb{R}_+^2 \setminus [\text{axes}].$
\end{enumerate}
Then $
\bZ$ is regularly varying on $
\E$ and has {\it hidden regular variation\/} (HRV) on $\E_0$ if there exist
$0<\alpha \leq \alpha_0$, scaling functions $b(t) \in
RV_{1/\alpha}$ and $b_0 \in RV_{1/\alpha_0}$ with $b(t)/b_0(t) \to
\infty$ and limit measures $\nu,\, \nu_0$ such that $$\bZ\in
\MRV(\alpha, b(t), \nu, \E) \cap \MRV(\alpha_0, b_0(t), \nu_0, \E_0)$$
so that unpacking the notation we get,
\begin{equation}\label{eq:regVarE}
t\Prob[\bZ / b(t) \in \cdot \,]\to \nu (\cdot), \quad \text{ in }\M(\E)
\end{equation}
and
\begin{equation}\label{eq:regVarE0}
t\Prob[\bZ / b_0(t) \in
 \cdot \,]\to \nu_0 (\cdot), \quad \text{ in
} \M(\E_0).
\end{equation}
 On $\E$ we may take
$\aleph_{\bzero}=\{\bx: \|\bx\|=1\}$ for a convenient choice of
  $d(\bx,\by)=\|\bx -\by\|$ and on $\E_0$,
  $$\aleph_{[\axes]}:=\{\bx \in \E: x_1\wedge x_2=1\}$$ is the
  appropriate unit sphere. 
 Then using $\gpolar$ \eqref{eq:regVarE} and
  \eqref{eq:regVarE0} become,
\begin{equation}\label{eq:regVarEPolar}
t\Prob\bigl[\bigl(\|\bZ \|/ b(t), \bZ/\|\bZ\|\bigr) \in \cdot \,\bigr]\to
\nu_\alpha  \times S (\cdot), \quad \text{ in } \M((0,\infty)\times
\aleph_{\bzero  })
\end{equation}
and
\begin{equation}\label{eq:regVarE0Polar}
t\Prob\Bigl[ 
\Bigl(\frac{Z_1\wedge Z_2 }{b_0(t)} , \frac{\bZ}{Z_1\wedge Z_2} \Bigr) \in \cdot \Bigr]
\to 
\nu_{\alpha_0} \times S_0 (\cdot) \qquad \text{ in }
\M\bigl((0,\infty)\times \aleph_{[\axes]}\bigr)
\end{equation}
and $S$ and $S_0$ are probability measures on $\aleph_{\bzero}$ and
  $\aleph_{[\axes]}$ respectively. Note
$$\Bigl(\frac{\bz}{z_1\wedge z_2} \Bigr) =\begin{cases}
(1,z_2/z_1),& \text{ if }z_1\leq z_2,\\
(z_1/z_2,1),& \text{ if }z_2<z_1
\end{cases} $$
and 
$$\aleph_{[\axes]}= \bigl([1,\infty)\times
\{1\}\bigr) \cup \bigl(\{1\}\times  [1,\infty)\bigr).$$
So we may rewrite \eqref{eq:regVarE0Polar} as two statements: For
$x\geq 1$,
\begin{align}
&t\Prob[ \frac{Z_1}{b_0(t)} >r, \frac{Z_2}{Z_1}>x] \to r^{-\alpha_0}
S_0\{(1,z):z>x\}=:r^{-\alpha_0}p\bar G_1 (x),\label{eq:G1}\\
&t\Prob[ \frac{Z_2}{b_0(t)} >r, \frac{Z_1}{Z_2}>x]\to r^{-\alpha_0}
S_0\{(z,1):z>x\}=:r^{-\alpha_0}q\bar G_2 (x),\label{eq:G2}
\end{align}
where {$p:=S_0 \{\{1\}\times [1,\infty)\},$     
$q:=S_0 \{[1,\infty) \times \{1\} \}=1-p$} and $G_1,G_2$ are  probability
distributions on $[1,\infty)$.  We also have
\begin{align}
&t\Prob[ \frac{Z_1\wedge Z_2}{b_0(t)} >r, \left(\frac{Z_1}{Z_2}
  \bigvee \frac{Z_2}{Z_1}\right)>x] \to 
r^{-\alpha_0}\left(p\bar G_1 (x)+q\bar{G}_2 (x)\right).\label{eq:Gmix}
\end{align}

Traditionally \citep{resnickbook:2007}, regular variation on $\E$ has been studied using the one
point uncompactification, vague convergence and the polar coordinate
transform $\bx\mapsto (\|\bx\|,\bx/\|\bx\|).$ On $\E$ this works fine
because $\{\bx \in \E: \|\bx\|=1\}$ is compact and lines through
$\binfty$ cannot carry mass.  However,  on $\E_0$ the
traditional unit sphere $\{\bx \in \E_0: \|\bx\|=1\}$ is no longer
compact. Hence, Radon measures on $\{\bx\in \E_0: \|\bx\|=1\}$ may
not be finite and for estimation problems
 the approach relying on vague convergence is a dead end if 
estimation of a possibly infinite measure is
required. More details on why an approach without compactification is
desirable are in \cite{hult:lindskog:2006a, das:mitra:resnick:2013,
  lindskog:resnick:roy:2013}. We emphasize it is difficult to
discuss MRV on $\E_0$ with the conventional unit sphere and it is
preferable to use $\aleph_{[\axes]}$.

\subsection{Basic notation}\label{subsec:notation}
Here is a  notation and concept summary.
$$ 
\begin{array}{llll}
RV_\beta & \text{Regularly varying functions with index $\beta>0$. We
  can and do}\\
&\text{assume such functions are continuous and strictly increasing.}\\[2mm]
\E & \mathbb{R}^2 \setminus \{\bzero\}.\\[2mm]
[\axes] & \{0\}\times \mathbb{R}_+ \cup \mathbb{R}_+ \times \{0\}.\\[2mm]
\E_0 &\mathbb{R}^2 \setminus [\axes].\\[2mm]
\M(\mathbb{C}\setminus \mathbb{C}_0) & \text{The set of all non-zero
measures on $\mathbb{C}\setminus \mathbb{C}_0$ which are finite on}\\&\text{
subsets bounded away from $\mathbb{C}_0$.}\\[2mm]
\mathcal{C}(\mathbb{C}\setminus \mathbb{C}_0) & \text{Continuous, bounded,
  positive functions on $\mathbb{C}\setminus \mathbb{C}_0$ whose
  supports}\\&\text{are bounded away from $\mathbb{C}_0$. Without
  loss of generality \citep{lindskog:resnick:roy:2013}, we
  may}\\&\text{assume the functions are
  uniformly continuous.}\\[2mm] 
\mu_n\to \mu & \text{Convergence in $\M(\mathbb{C}\setminus
  \mathbb{C}_0)$ means $\mu_n (f) \to \mu(f)$ for all}\\
 &f \in \mathcal{C}(\mathbb{C}\setminus \mathbb{C}_0). \text{ See
   \cite{hult:lindskog:2006a, das:mitra:resnick:2013,
  lindskog:resnick:roy:2013}.}
\\[2mm]
\aleph_{\mathbb{C}} & \{\bx : d(\bx,\mathbb{C} )=1\}.\\[2mm]
\aleph_{\bzero} & \{\bx \in \E : d(\bx, \{\bzero \})=1\}.\\[2mm]
\aleph_{[\axes]} & \{\bx \in \E_0: d(\bx,[\axes])=1\}=\{1\}\times
[1,\infty)\cup [1,\infty)\times \{1\}.\\[2mm]
\MRV & \text{multivariate regular variation; for this paper, it means regular variation
  on $\E$}.\\[2mm]
\HRV & \text{hidden regular variation; for this paper, it means regular variation on
  $\E_0$}.\\[2mm]
\gpolar & \text{{Polar co-ordinate transformation} relative to the deleted cone $\mathbb{C}_0$,}\\ 
&\gpolar(\bx)=\bigl(d(\bx, \mathbb{C}_0),
\bx/d(\bx,\mathbb{C}_0)\bigr). \text{ See
  \cite{lindskog:resnick:roy:2013, das:mitra:resnick:2013}.}\\[2mm]
\bX \independent \bY & \text{The random elements $\bX, \bY$ are independent.}
\end{array}
$$

\section{Generating Regularly Varying Models} \label{sec:genMRV}
We outline schemes for generating regular variation. These schemes
generate the full totality of asymptotic limits but not the full
totality of pre-asymptotic models; {so there can be many other ways to get the same asymptotic models.}
\subsection{Generating regular variation on
  $\E$.}\label{subsec:regVarE}
The easiest way to obtain a regularly varying model on $\E$ 
with scaling function $b(t)$ and limit measure 
$\nu(\cdot) =\nu_\alpha \times S \circ {\gpolar}$ is as follows:
Suppose $R$ is a random element of $(0,\infty)$ with a regularly
varying tail and scaling function $b(t)$:
$$t\Prob[R/b(t) >x]\to x^{-\alpha},\quad x>0,\,\alpha>0.$$
Let $\bTheta$ be a random element of $\aleph_{\bzero }$ with
distribution $S$
$$\Prob[\bTheta \in \cdot \,] =S(\cdot)$$ and which is independent of
$R$. 
Then $\bZ:=R\bTheta =  \gpolar^\leftarrow (R,\bTheta)           $ is regularly varying on $\E $ with limit measure
$\nu=\nu_\alpha \times S \circ {\gpolar}$ on $\E$
because
\eqref{eq:regVarEPolar} and consequently \eqref{eq:regVarE} hold. Note
$\gpolar $ is defined relative to the deleted cone $\{\bzero\}$.

\subsection{Generating regular variation on
  $\E_0$ (and sometimes also on $\E$).}\label{subsec:regVarE0}
As suggested in \cite{maulik:resnick:2005}, we may follow the same
scheme as in Section \ref{subsec:regVarE}. Let $R_0 $ be a random
element of $(0,\infty)$ that is regularly varying with index
$\alpha_0$ and scaling function $b_0(t)$. Let $\bTheta _0$ be a random
element of $\aleph_{[\axes]}$ with distribution $S_0$ and independent
of $R_0$. Then
$\bZ=R_0\bTheta_0
=\gpolar^\leftarrow (R_0,\bTheta_0)$ is regularly varying with scaling
function $b_0(t)$ and limit measure $\nu_0:=\nu_{\alpha_0} \times S_0
\circ \gpolar^{-1}$ on $\E_0$ because \eqref{eq:regVarE0Polar} and therefore
\eqref{eq:regVarE0} hold.

In practice we specify the measure $S_0$ on $\aleph_{[\axes]}$ as
follows: Let $G_1,G_2$ be two probability measures on $(1,\infty)$ and
define
\begin{equation}\label{eq:mixThetas}
\bTheta_0=B(\Theta_1,1)+(1-B)(1,\Theta_2)
\end{equation}
 where
$B,\Theta_1,\Theta_2$ are independent, $B$ is a Bernoulli switching
variable with $P[B=1]=p=1-P[B=0]$ and $\Theta_i$ has distribution
$G_i$, $i=1,2$.
So $G_1$ smears probability mass on the horizontal line emanating from
$(1,1)$ and $G_2$ does the same thing for the vertical line.

For estimation purposes, note for $s>1$ that
\begin{align}
\bar G_1(s)&=G_1(s,\infty) =\nu_0\{\bx \in \E_0: x_1/x_2>s\},\label{eq:G1new}\\
\bar G_2(s)&=G_2(s,\infty) =\nu_0\{\bx \in \E_0: x_2/x_1>s\}.\label{eq:G2new}
\end{align}

Depending on the moments of $G_i$, $i=1,2$, it may be possible to
extend the regular variation constructed  on $\E_0$ to $\E$ so that the marginals
$Z_1,Z_2$ individually have tails which are regularly varying. This
means \citep{maulik:resnick:2005}
$$\nu_0\{\bx\in\E_0: \|\bx\|>1\}<\infty,$$
which occurs when 
$$\bigvee_{i=1}^2 \int_1^\infty s^{\alpha_0  -1} \bar G_i(s)ds
<\infty,$$
and is thus a somewhat restricted case.
Regular variation on $\E_0$ by itself does not in general imply one dimensional
regular variation of the marginals. Also if the tails of $G_i$ are
heavier than the tail of $R$, we can have regular variation on $\E_0$ with
index $\alpha_0$ but the tails of $Z_1$ and $Z_2$ may be regularly varying
with a smaller index $\alpha$. Full discussion is in
\cite{maulik:resnick:2005}.

\section{Generating models that have both multivariate regular
  variation on $\E$ and HRV on $\E_0$.}\label{sec:genMRVHRV}
We summarize  several methods for generating models possessing both MRV on
$\E$ and HRV on $\E_0$.

\subsection{Mixture method.}\label{subsec:methodMix}
This method \cite{maulik:resnick:2005, resnickbook:2007} expresses the
random vector $\bZ$ as 
$$\bZ=B\bY+ (1-B)\bV,
$$
a mixture where $\bY$ gives the regular variation on $\E$ and $\bV$
gives the regular variation on $\E_0$. Since HRV implies that MRV on
$\E$ must include asymptotic independence \citep{resnick:2002a,
  resnickbook:2007}, we need $\bY$ to model MRV with index $\alpha$ on
$\E$ and have asymptotic independence. So we take $\bY$ to concentrate
on $[\axes]$ and 
\begin{equation}\label{eq:YhugAxes}
\bY=B_1 (\xi_1,0) +(1-B_2)(0,\xi_2)
\end{equation}
 where $B_1,\xi_1,\xi_2 $ are
independent, $B_1$ is a Bernoulli switching variable and 
\begin{equation}\label{eq:marginalsRegVar}
t\Prob[\xi_i/b(t)>x]\to x^{-\alpha},\quad x>0, \,\alpha >0, \,t\to\infty.
\end{equation}
Construct $\bV $ by the scheme of Section \ref{subsec:regVarE0} to be
regularly varying on $\E_0$ with limit measure $\nu_0$ and scaling 
function $b_0(t)$.  The resulting $\bZ$ has both MRV on $\E $ 
and HRV on $\E_0$:
$$\bZ \in \MRV(\alpha,b(t),\nu, \E)\cap \MRV(\alpha_0,b_0(t),\nu_0, \E_0).$$

\subsection{Additive method.} Weller and Cooley
\citep{weller:cooley:2013} advocate  an additive model of the form 
$$\bZ=\bY+\bV,$$
where $\bY\in \MRV(\alpha,b(t),\nu,\E)$ 
and $\bV$ has HRV and $\bV \in \MRV(\alpha_0 , b_0(t), \nu_0,
\E_0)$ and $\bY\independent \bV$.  
They argue there are advantages for estimating the parameters and the
additive model overcomes the undesireable and usually unrealistic
feature of the mixture method that points are 
installed  directly on the axes.
 However, as we will see, the additive model does not
always successfully separate the HRV piece in a way that is
identifiable.

\subsubsection{Simple case: $\bY$ has no HRV and there is a  finite
  hidden angular   measure.}\label{subsubsec:finiteAngMeas} 
We start with the simplest result.

\begin{prop}\label{prop:1} 
Suppose 
\begin{enumerate}
\item $\bY$ has the structure given in
  \eqref{eq:YhugAxes} (so that $\bY$ has no HRV)  and 
  \eqref{eq:marginalsRegVar} holds.
\item $\bV$ has MRV on $\E$ (not $\E_0$) with index $\alpha_0\geq
  \alpha$, scaling function $b_0(t) = o(b(t))$, limit measure $\nu_0
  \in M(\E)$ and no asymptotic
  independence. 
Regular variation
  of $\bV$ on $\E$ has the consequence that  for $i=1,2$,
\begin{equation}\label{eq:marginalV}
t\Prob[V_i>b_0(t)x]\to c_ix^{-\alpha_0} ,\,x>0,
  \,t\to\infty, \,c_i\geq 0,\, c_1\vee c_2>0.\end{equation}
\end{enumerate}
Then $\bZ:=\bY+\bV$ has 
\begin{enumerate}
\item MRV on $\E$: $\bZ \in \MRV(\alpha,b(t),\nu,\E)$  and
$\bZ$  has asymptotic independence.
\item HRV on $\E_0$:  $\bZ \in \MRV(\alpha_0,b_0(t),{\nu_0}_{|_{\E_0}}
  , \E_0)$. The limit measure  ${\nu_0}_{|_{\E_0}}$ is 
 $\nu_0 $ restricted to $\E_0$ and 
\begin{equation}\label{eq:nu0finite}
\nu_0\{\bx\in \E_0: \|\bx\|\geq 1\}<\infty.
\end{equation}
 \end{enumerate}
\end{prop}

The last condition means the hidden limit measure $\nu_0 $ has finite
spectral measure with respect to the conventional unit sphere
since $\bV$ has MRV on $\E$.
So the construction in 
Proposition \ref{prop:1} yields only a special case of HRV since there
are many cases where \eqref{eq:nu0finite} fails.

\begin{proof}
The statement about MRV on $\E$ can be deduced from known results,
eg. \citet[p. 230]{resnickbook:2007}, \citet{jessen:mikosch:2006,
  resnick:1986}. (Note, it would not be enough to assume $\bV \in
\MRV(\alpha_0, b_0, \nu_0, \E_0)$.)
To prove HRV of $\bZ$ on $\E_0$,
we apply criterion (ii) of the Portmanteau Theorem 2.1 in
\cite{lindskog:resnick:roy:2013} and 
 let $f \in \mathcal{C}((0,\infty)^2)$ and
without loss of generality suppose $f$ is bounded by a constant $\|f\|$, uniformly
continuous and 
$$f(\bx)=0, \quad \text{ if }x_1 \wedge x_2<\eta,$$
for some $\eta >0$. Uniform continuity of $f$ means that the modulus
of continuity 
$$\omega_f(\delta):=\sup \{|f(\bx)-f(\by)|: \|\bx -\by\| <\delta\} \to
0,\quad (\delta \to 0).$$
Since $\bV$ has MRV on $\E$ we have
$$t\bfE f(\bV/b_0(t)) \to \nu_0(f),$$
and so it suffices to show as $t\to\infty$.
\begin{equation}\label{eq:ditchY}
t\bfE f\Bigl(\frac{\bY+\bV}{b_0(t)} \Bigr) -t\bfE f\Bigl(\frac{\bV}{b_0(t)}\Bigr) \to 0.
\end{equation}
Because of the special structure of $\bY$, the absolute value of the
difference on the previous line is bounded 
by 
\begin{align*}
\frac t2 \bfE\Bigl |  f\Bigl(
&\frac{\xi_1+V_1}{b_0(t)},\frac{V_2}{b_0(t)}\Bigr) -f(\bV/b_0(t))\Bigr|
+
\frac t2 \bfE\Bigl |  f\Bigl(
\frac{V_1}{b_0(t)},\frac{\xi_2+V_2}{b_0(t)}\Bigr) -f(\bV/b_0)\Bigr|\\
=&I+II.
\end{align*}
For $\delta <\eta$, write
$$2I =t\bfE|\cdot|1_{[\xi_1/b_0(t) <\delta]} + t\bfE|\cdot|1_{[\xi_1/b_0(t) >\delta]} =2Ia+2Ib.$$
To keep both terms of the difference from being zero we  write
\begin{align*}
2Ia=
t \bfE\Bigl |  f\Bigl(
&\frac{\xi_1+V_1}{b_0(t)},\frac{V_2}{b_0(t)}\Bigr)
-f(\bV/b_0)\Bigr|1_{[\xi_1<b_0(t)\delta, V_1>b_0(t) (\eta -\delta)]}\\
\leq & \omega_f(\delta ) t\Prob[V_1>b_0(t) (\eta -\delta)] \to \omega_f
(\delta ) c_1 (\eta -\delta)^{-\alpha_0} \quad (t\to\infty),\\
\to & 0 \quad (\delta \to 0),
\end{align*}
where we used \eqref{eq:marginalV}.

For $2Ib$, in order to keep both terms of the difference from being
zero, we write,
\begin{align*}
2Ib=&\frac t2 \bfE \Bigl |  f\Bigl(
\frac{\xi_1+V_1}{b_0(t)},\frac{V_2}{b_0(t)}\Bigr)
-f(\bV/b_0)\Bigr|1_{[\xi_1 >b_0(t)\delta, V_2>b_0(t)\delta]}\\ 
\leq &2\|f\| t\Prob[V_2>b_0(t)\delta]\Prob[\xi_1>b_0(t)\delta ]\\
\intertext{and as $t\to\infty$ this is}
\sim &2\|f\| c_2 \delta^{-\alpha_0}\Prob[\xi_1> b_0(t)\delta] \to 0 \quad
(t\to\infty).
\end{align*}
We handle $II$ similarly.\end{proof}

\begin{exm}\label{eg:simple}
Suppose $\bY$ has the structure given in
  \eqref{eq:YhugAxes} where $\xi_1,\xi_2$ are iid Pareto distributed with index
  $\alpha$. Assume $\bV=R_0\bTheta _0$ where $R_0$ is Pareto distributed index $\alpha_0
  >\alpha$
and $\bTheta_0 $ has the structure given  in \eqref{eq:mixThetas} where
$\Theta_i=1+\bfE_i$ and $E_1,E_2 $ are two standard iid exponential
random variables. Then $\bV=R_0\bTheta_0 \in \MRV(\alpha_0,b_0(t), \nu_0,
\E)$ and 
$$\nu_0=\nu_{\alpha_0} \times \Prob[\bTheta \in \cdot\,] \circ \gpolar^{-1}.$$
This construction makes the marginals of $\bV=(V_1,V_2) $ regularly
varying with index $\alpha_0$ which is consistent with $\bV$ being MRV
on $\E$ rather than just $\E_0$:
\begin{align*}
\Prob[V_1>x]=&p\Prob[R(1+E_1)>x]+q\Prob[R>x] \\
\sim & px^{-\alpha_0} \mathbf{E}
\bigl((1+E_1)^{\alpha_0} \bigr) +qx^{-\alpha_0},\\
\intertext{(where the $\sim $ results from an application of Breiman's
  theorem \citep{breiman:1965} on products)}
=& (const) x^{-\alpha_0}.
\end{align*}
{Here $p= 1-q=P(\bTheta_0 \in ((1,\infty)\times\{1\}))$.}

To check  whether we can get
identify the distributions of $\bY$ and $\bV$ from a data sample of
$\bZ=\bY+\bV$, we simulate data following this model for three
different choices 
of $\alpha$ while keeping $\alpha_0$ fixed. We then check whether we can estimate
back the values of $\alpha$ and $\alpha_0$. In all the three cases
$\alpha_0=2$ with $\Theta_1\eqdis \Theta_2 $ with $(\Theta_1-1)$
following an iid standard exponential distribution and $p=0.5$. In
each case we simulate 10000 iid samples from $\bZ$. Then we create
Hill plots for the marginals of $Z_1$ and $Z_2$ to identify the value
of $\alpha$. To detect the hidden part we create a Hill plot for
$\min(Z_1,Z_2)$ to find the value of $\alpha_0$. Referencing
\eqref{eq:Gmix}, we also make a QQ plot of $\max(Z_1/Z_2,Z_2/Z_1)$ for
the 100 highest values of $\min (Z_1,Z_2)$ against the quantiles of
standard exponential which is the distribution of $\Theta_1$ and
$\Theta_2$. We discuss the cases below.

\begin{enumerate}
\item[Case 1:]  $\alpha=1$. The top panel of Figure \ref{fig:Ex31a}  indicates that we can identify the tails of $\bZ$
to be heavy tailed. The correct index $\alpha=1$ is slightly overestimated. The Hill plot of $\min(Z_1,Z_2)$ also indicates HRV on $\E_0$ with index close to
$\alpha_0=2$. The QQ plot of  $\max(Z_1/Z_2,Z_2/Z_1)$ thresholded by
the 100 largest  values of $\min (Z_1,Z_2)$  against standard
exponential shows a decent fit. 

  \begin{figure}[h]
  \begin{centering}
  \includegraphics[width=10cm]{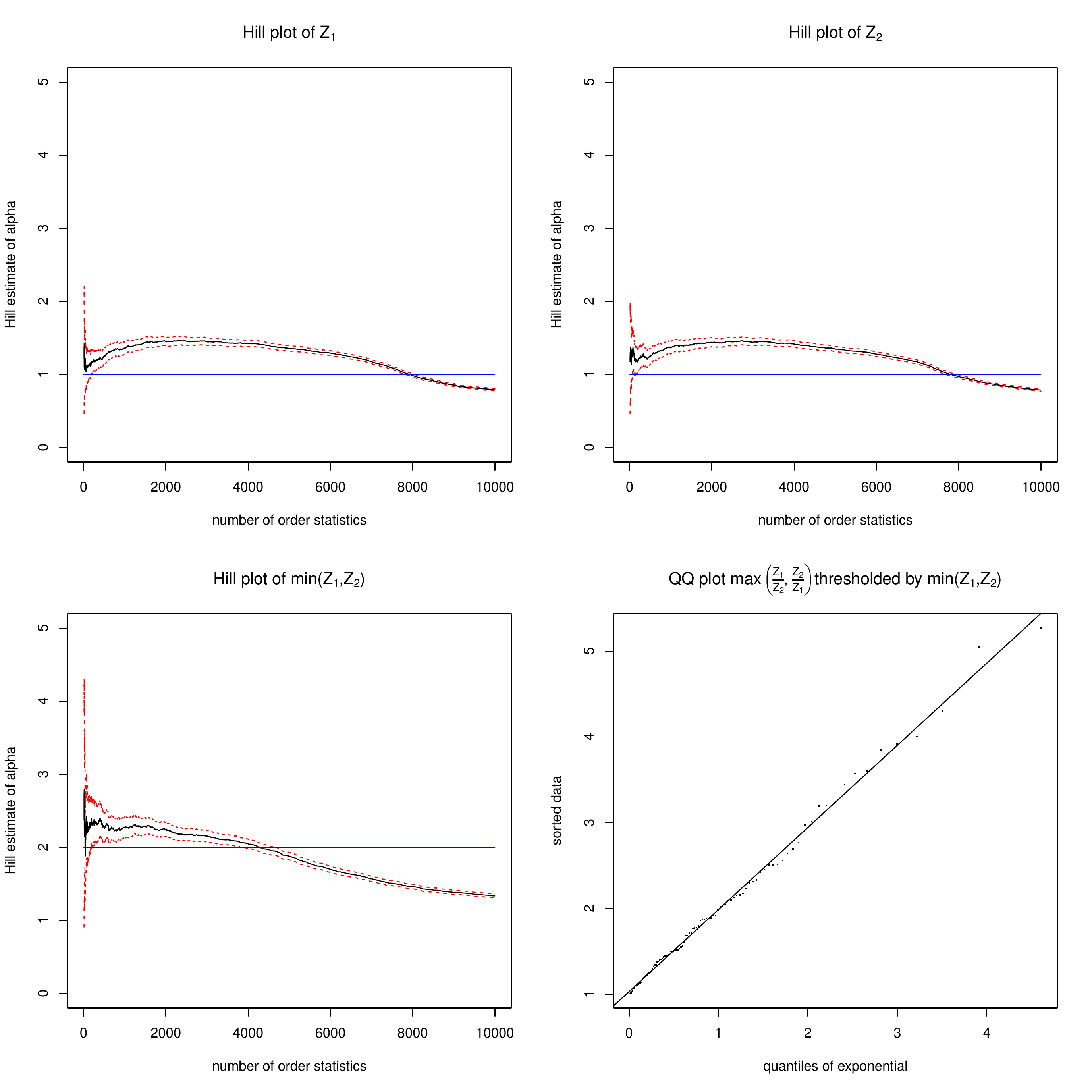}
  \end{centering}
  \caption{Exploratory plots for Example \ref{eg:simple}, case 1, with
    $\alpha=1, \alpha_0=2$. {\it Top panel:} Hill plots for the
    marginals $Z_1$ and $Z_2$. {\it Bottom left:} Hill plot for
    $\min{\{Z_1,Z_2\}}$. {\it Bottom right:} exponential QQ plot of
    $\max{\{Z_1/Z_2,Z_2/Z_1\}}$ thresholded by the 100 largest values of
    $\min \{Z_1,Z_2\}$.} \label{fig:Ex31a} 
  \end{figure}

  \item[Case 2:]  $\alpha=1.5$. The top panel of Figure \ref{fig:Ex31b}  again indicates that we can identify the tails of $\bZ$
to be heavy tailed. The index $\alpha$ is again overestimated, this
time more than in the previous case, perhaps because of the closeness of $\alpha$ to $\alpha_0$. The Hill plot of $\min\{Z_1,Z_2\}$ also indicates HRV on $\E_0$ with index close to
$\alpha_0=2$. The QQ plot of  $\max\{Z_1/Z_2,Z_2/Z_1\}$ thresholded by
the 100 largest values of $\min \{Z_1,Z_2\}$  against standard
exponential shows a decent fit again. 

  \begin{figure}[h]
  \begin{centering}
  \includegraphics[width=10cm]{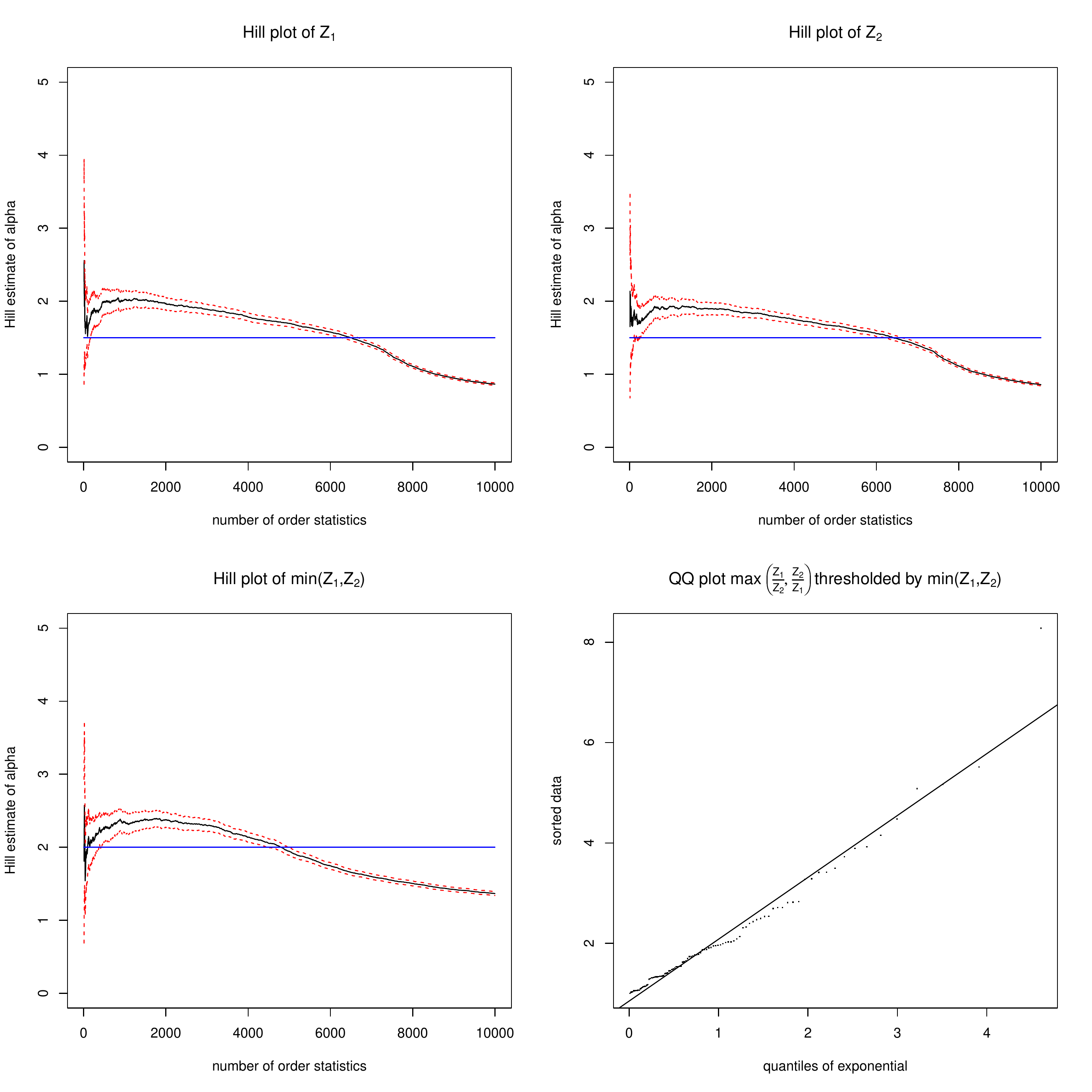}
   \end{centering}
  \caption{Exploratory plots for Example \ref{eg:simple}, case 2, with
    $\alpha=1.5, \alpha_0=2$. {\it Top panel:} Hill plots for the
    marginals $Z_1$ and $Z_2$. {\it Bottom left:} Hill plot for
    $\min\{Z_1,Z_2\}$. {\it Bottom right:} exponential QQ plot of
    $\max\{Z_1/Z_2,Z_2/Z_1\}$ thresholded by the 100 largest values of
    $\min \{Z_1,Z_2\}$.} \label{fig:Ex31b} 
  \end{figure}
  
 \medskip \item[Case 3:]  $\alpha=0.5$. In this case too, the top panel of
    Figure \ref{fig:Ex31c}  indicates heavy tailed behavior of
    $\bZ$. The Hill plot of $\min(Z_1,Z_2)$ also indicates hidden
    regular variation. The indices $\alpha=0.5$ and $\alpha_0 =2$ are
 reasonably estimated here, presumably because the original
    values of $\alpha$ and $\alpha_0$ are far apart. However, 
    the exponential QQ plot of  $\max\{Z_1/Z_2,Z_2/Z_1\}$ for the 100
largest values of $\min \{Z_1,Z_2\}$  struggles to indicate an
    exponential fit. 
  \begin{figure}[h]
  \begin{centering}
  \includegraphics[width=10cm]{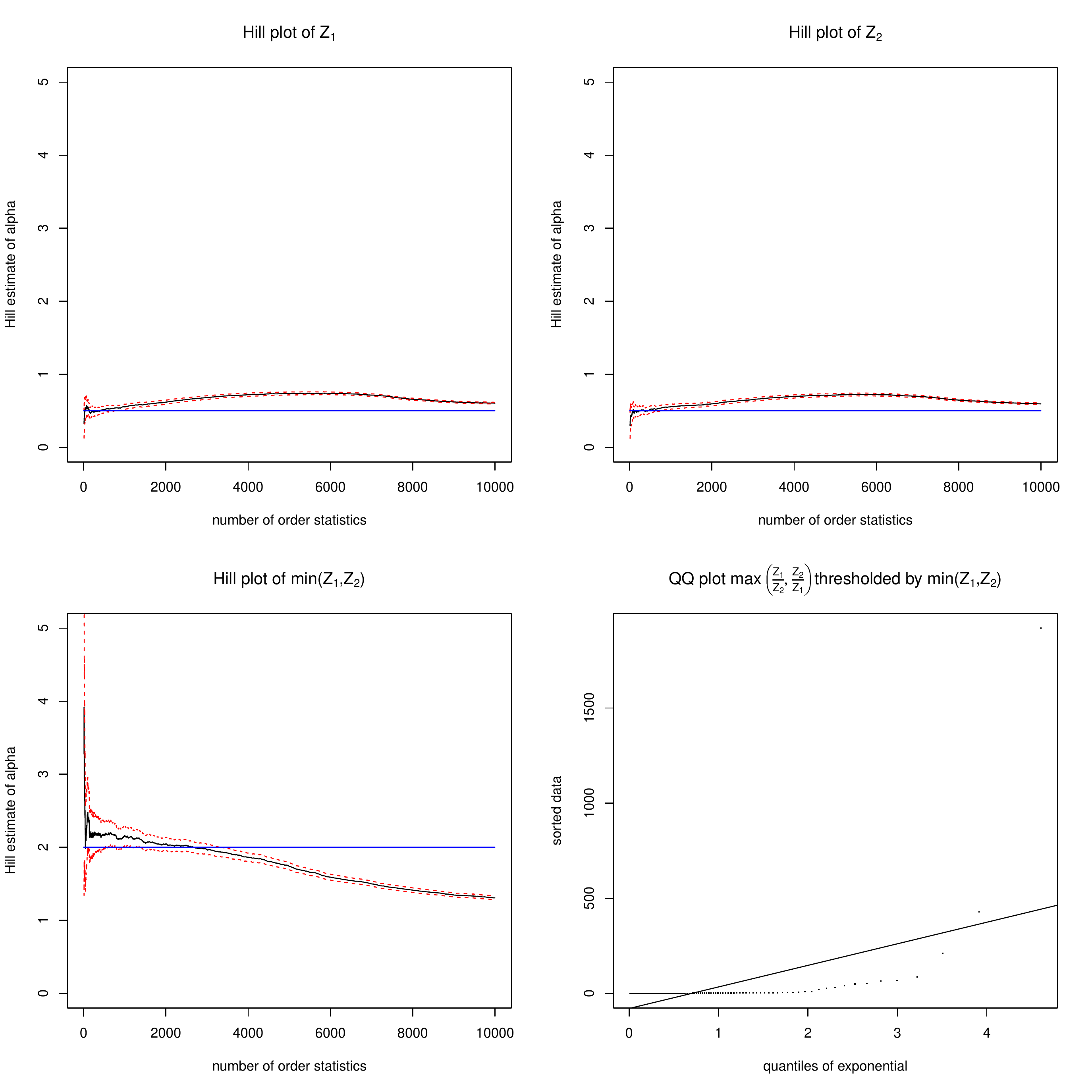}
  \end{centering}
  \caption{Exploratory plots for Example \ref{eg:simple}, case 3, with
    $\alpha=0.5, \alpha_0=2$. {\it Top panel:} Hill plots for the
    marginals $Z_1$ and $Z_2$. {\it Bottom left:} Hill plot for
    $\min\{Z_1,Z_2\}$. {\it Bottom right:} exponential QQ plot of
    $\max\{Z_1/Z_2,Z_2/Z_1\}$ thresholded by the 100 maximum values of
    $\min (Z_1,Z_2)$.} \label{fig:Ex31c} 
  \end{figure}

\end{enumerate}
\qed \end{exm}

\subsubsection{What happens if $\bY$ has HRV but $\bV$ has no
  HRV}\label{subsubsec:Yhrv} 
In Proposition \ref{prop:1}, we can remove the restriction that $\bY=(Y_1,Y_2)$
concentrates on the axes at the expense of a tail condition on $\bY$
that guarantees the tails of $\bV $ and $\bY$ do not interact in such
as way as to obscure the fact that the hidden angular measure of $\bZ$
is that of $\bV$.
Continue to suppose {$\bZ=\bY+\bV$ with $\bY\independent \bV$}.
\begin{prop}\label{prop:2} Suppose
\begin{enumerate}
\item $\bY\in \MRV(\alpha, b(t), \nu, \E)$ and {exhibits} asymptotic independence.
\item $\bV$  has  \MRV on $\E$ (not $\E_0$) with index $\alpha_0 \geq
  \alpha$, scaling function $b_0(t)=o(b(t))$, limit measure $\nu_0 \in
  M(\E)$
  with no asymptotic independence so that
$$t\Prob[\bV/b_0(t) \in \cdot \,] \to \nu_0 \quad \text{ in } \M(\E).$$
\item The interaction of the tails of $\bY$ and $\bV$ is  controlled by
  the condition
\begin{equation}\label{eq:nononanette}
t\Prob[Y_1\wedge Y_2>b_0(t) x] \to 0,\quad t\to\infty, \,x>0.
\end{equation}
\end{enumerate}
Then $\bZ=\bY+\bV$ has
\begin{enumerate}
\item $\MRV(\alpha, b(t), \nu, \E)$ and
  asymptotic independence.
\item HRV on $\E_0$ with index $\alpha_0$, scaling function $b_0(t)$,
  limit measure $\nu_0$ restricted to $\E_0$.
  \end{enumerate}
\end{prop}

{\bf Remarks:} 
For the $\bY$ defined in Proposition \ref{prop:1}, $Y_1\wedge Y_2=0$
so \eqref{eq:nononanette} is automatic. If $Y_1,Y_2$  are iid with
$\Prob[Y_i>x] \in RV_{-\alpha}$,  $\bY$ itself has HRV \cite{resnick:2002a,
  resnickbook:2007} with index $\alpha_0=2\alpha$
and condition \eqref{eq:nononanette} is needed to guarantee the
HRV of $\bZ$ comes from $\bV$ and not $\bY$. Condition
\eqref{eq:nononanette} is equivalent in this case to 
\begin{equation}\label{eq:ratios}
\frac{(\Prob[Y_1>x])^2}{\Prob[V_1\wedge V_2>x]} \to 0, \quad (x\to\infty).
\end{equation}
and it is sufficient that
$$\frac{\alpha_0}{2}<\alpha <\alpha_0.$$
This is seen by noting that for $Y_1,Y_2$ iid index $\alpha$, \eqref{eq:nononanette} is
\begin{align*}
t\Bigl( 
\Prob[Y_1>b_0(t)x]\Bigr)^2=&
t\Bigl( \Prob[Y_1>b\bigl( b^\leftarrow(b_0(t))\bigr)x]\Bigr)^2\\
=& \frac{t}{ b^\leftarrow(b_0(t))}\Bigl( b^\leftarrow(b_0(t))
\Prob[Y_1>b\bigl( b^\leftarrow(b_0(t))\bigr)x]\Bigr)^2\\
\intertext{and since $ b^\leftarrow(b_0(t)) \to \infty$ and $b(\cdot)$
  is the scaling function of $Y_1$, this is
  asymptotic to}
\sim & \frac{t}{ b^\leftarrow(b_0(t))} x^{-2\alpha}.
\end{align*}
We need $\lim_{t\to\infty} {t}/{ b^\leftarrow(b_0(t))}  =0$ and
unwinding this condition yields \eqref{eq:ratios}. 

\begin{proof} As in Proposition \ref{prop:1}, we focus on the HRV claim.
Again assume $f\in \mathcal{C}((0,\infty)^2) $
and $f$ is bounded by $\|f\|$, uniformly
continuous with
$$f(\bx)=0, \quad \text{ if }x_1 \wedge x_2<\eta,$$
for some $\eta >0$. We  need to show \eqref{eq:ditchY}. For any
small $\delta>0$ {with $\delta<\eta$}, the
absolute value of the difference in \eqref{eq:ditchY} is 
$$
tE|\cdot|1_{[Y_1\vee Y_2 >b_0(t) \delta]}
+
tE|\cdot|1_{[Y_1\vee Y_2 <b_0(t) \delta, V_1\wedge V_2>b_0(t)(\eta
  -\delta)]} =I+II, 
$$
since for the second term, the only way the difference can be non-zero
is if $\bV$ is sufficiently large. Term II is dominated by
\begin{align*}
{II} \leq & \omega_f (\delta) t \Prob[ V_1\wedge V_2>b_0 (t)(\eta-\delta)]\\
\sim & \omega_f(\delta ) (const) (\eta- \delta)^{-\alpha_0}, \quad
(t\to\infty)\\
\to & 0, \qquad \text{ as }\delta \to 0.
\end{align*}
For $I$ we have
\begin{align*}
I\leq &tE|\cdot|\Bigl(1_{[Y_1\wedge Y_2> b_0(t)\delta]]}
+1_{[Y_1> b_0(t)\delta, Y_2\leq b_0(t)\delta]]}
+1_{[Y_2>b_0(t)\delta, Y_1< b_0(t)\delta]]}
\Bigr)
\\
=&Ia + Ib + Ic.\end{align*}
The term $Ia$ can be quickly killed,
$$Ia \leq 2 \|f\|  t\Prob[Y_1\wedge Y_2> b_0(t) \delta] \to 0, \quad
(t\to\infty)$$
from \eqref{eq:nononanette}. The term $Ib$ is dominated by 
\begin{align*}
Ib\leq & 2\|f\| t \Prob[Y_1 > b_0(t) \delta , V_2> b_0(t)(\eta -\delta)]\\
=  & 2\|f\| t \Prob[ V_2> b_0(t)(\eta -\delta)] \Prob[Y_1 > b_0(t) \delta ]\\
\sim & \|f\| (\eta -\delta)^{-\alpha_0} \Prob[Y_1 > b_0(t) \delta ]\quad
(t\to\infty),\\
\to & 0 \qquad (t\to\infty).
\end{align*}
Term $Ic$ is handled similarly.
\end{proof}

\subsubsection{What happens if $\bY$ has no HRV but  $\bV$ has HRV}
A problem with the additive model is the tail weights contributing to
MRV on $\E$ and HRV on $\E_0$ can be confounded between $\bY$ and
$\bV$ and it is possible for $\bV$ to have MRV on $\E$, 
HRV on $\E_0$ but the hidden measure of $\bZ=\bY+\bV$ is not the 
hidden measure of $\bV$.

To focus on the influence of $\bV$, we again assume $\bY$ has the
structure \eqref{eq:YhugAxes} used in Proposition \ref{prop:1}.

\begin{prop}\label{prop:3}
Suppose
\begin{enumerate}
\item $\bY$ has form \eqref{eq:YhugAxes} where $\xi_1,\xi_2$ are iid,
  each with distributions having regularly varying tails with index
  $\alpha$ and scaling function $b(t)$.
\item $\bV$ has both MRV on $\E$ and HRV on $\E_0$:
\begin{enumerate}
\item $\bV \in  \MRV(\alpha_*, b_*(t), \nu, \E)$ and has asymptotic independence.
\item $\bV \in \MRV(\alpha_0, b_0(t), \nu_0,\E_0)$.
\end{enumerate}
\item The parameters $\alpha, \,\alpha_*,\,\alpha_0$ are related by 
$\alpha\leq \alpha_* \leq \alpha_0$ and the scaling functions $b(t),
b_*(t), b_0(t) $ satisfy $b_*(t)=o(b(t)), \, b_0(t)=o(b_*(t))$.
\item Define a scaling function $h(t)$ through its inverse
  $h^\leftarrow(t)$ by
\begin{equation}\label{eq:defh}
h^\leftarrow (t) =:b^\leftarrow (t) b_* ^\leftarrow (t) \sim (const)\frac{1}{\Prob[\xi_1>t]\Prob[V_1>t]}.
\end{equation}
\end{enumerate}
Then
\begin{enumerate}
\item\label{item:case1} If 
\begin{equation}\label{eq:2heavy}
h(t)/b_0(t) \to \infty,
\end{equation}
 $\bZ \in  \MRV(\alpha, b(t),\nu, \E)$ with asymptotic independence 
   and has HRV on $\E_0$ with index $\alpha+\alpha_*$ and 
 limit measure (different than the
  hidden measure of $\bV$):
\begin{equation}\label{eq:notNu0}
\nu_{\bZ, \text{hidden}}:=\frac 12 \bigl(\nu_\alpha\times
\nu_{\alpha_*} +  \nu_{\alpha_*} \times
 \nu_\alpha\bigr).
\end{equation}
A sufficient condition for \eqref{eq:2heavy} is $\alpha_*<\alpha_0 - \alpha.$
\item If 
\begin{equation}\label{eq:2light} 
h(t)/b_0(t) \to 0,
\end{equation}
then $\bZ \in \MRV(\alpha,b(t), \nu, \E) \cap
  \MRV(\alpha_0,b_0(t),\nu_0,  \E_0)$ and
$\bZ$  has asymptotic independence and has HRV
and the hidden limit measure $\nu_0$ of $\bZ$ is  the hidden measure of $\bV$.
A sufficient condition for \eqref{eq:2light} is
$\alpha_* >\alpha_0 -\alpha $
\item If 
\begin{equation}\label{eq:justRight} 
h(t)/b_0(t) \to c \in (0,\infty),
\end{equation}
then $\bZ \in  \MRV(\alpha,b(t),\nu, \E$) with asymptotic independence
and $\bZ$ has HRV with index
$\alpha+\alpha_*$ and hidden measure which is a sum of the
measure given in \eqref{eq:notNu0} and $\nu_0$, the hidden
measure of $\bV$,
\begin{align}\label{eq:mixallnus} 
 \nu_{\bZ}& =\frac 12 \bigl(\nu_\alpha\times
\nu_{\alpha_*} +  \nu_{\alpha_*} \times
 \nu_\alpha\bigr) + \nu_0.
\end{align}
A sufficient condition  for \eqref{eq:2light} is
$\alpha_*=\alpha_0 -\alpha.$
\end{enumerate}
\end{prop}

\begin{proof}
Begin with the following observations for all cases: As $t\to\infty$,
\begin{align}
&t\Prob[\Bigl(\frac{\xi_1}{h(t)},  \frac{V_2}{h(t)} \Bigr) \in \cdot \,] \to
\nu_\alpha \times \nu_{\alpha_*}\label{eq:obs1}\\
&t\Prob[\Bigl( \frac{V_1}{h(t)} ,\frac{\xi_2}{h(t)}\Bigr) \in \cdot \,] \to
\nu_{\alpha_*} \times \nu_{\alpha}\label{eq:obs2}
\end{align}
in $\M((0,\infty)^2)$. To see this, write for $x>0,\,y>0$,
\begin{align*}
t\Prob[\xi_1>h(t)x, &V_2>h(t)y]= t  \Prob[\xi_1>b\circ b^\leftarrow (h)x]
 \Prob[V_2>b_*\circ b_*^\leftarrow (h)y]\\
=&\frac{t}{ b^\leftarrow (h)  b_*^\leftarrow (h)             }
b^\leftarrow (h)  \Prob[\xi_1>b\circ b^\leftarrow (h)x]\\
&\qquad \qquad b_*^\leftarrow (h) \Prob[V_2>b_*\circ b_*^\leftarrow (h)y]\\
\sim &
\frac{t}{ b^\leftarrow (h)  b_*^\leftarrow (h)             } x^{-\alpha}y^{-\alpha_*}\\
\sim &\nu_\alpha (x,\infty) \nu_{\alpha_*} (y,\infty).
\end{align*}
The proof of \eqref{eq:obs2} is the same.

Now assume $f\in \mathcal{C}((0,\infty)^2) $
and $f$ is bounded by $\|f\|$, uniformly
continuous with
$$f(\bx)=0, \quad \text{ if }x_1 \wedge x_2<\eta,$$
for some $\eta >0$.  Write
\begin{equation}\label{eq:pieces}
tEf\Bigl(\frac{\bY+\bV}{h(t)}\Bigr)=
\frac t2Ef\Bigl(\frac{\xi_1+V_1}{h(t)}, \frac{V_2}{h(t)}\Bigr)
+\frac t2 Ef\Bigl(\frac{V_1}{h(t)}, \frac{\xi_2+V_2}{h(t)}\Bigr) =A+B
\end{equation}

For case (1) where \eqref{eq:2heavy} holds, we get a limit for $A$ by writing
\begin{align*}
tE\Bigl|f\Bigl(\frac{\xi_1+V_1}{h(t)}, \frac{V_2}{h(t)}\Bigr)
-&f\Bigl(\frac{\xi_1}{\h},\frac{V_2}{\h}\Bigr) \Bigr|
=tE\Bigl| \cdot \Bigr| 1_{[V_1<\h \delta, \xi_1>\h (\eta -\delta),
  V_2>\h \eta ]}\\
&+tE\Bigl| \cdot \Bigr| 1_{V_1>\h \delta,  V_2>\h \eta ]} =I+II.
\end{align*}
Now 
\begin{align*}
I \leq &  \omega_f(\delta) t \Prob[ \xi_1>\h (\eta -\delta),
  V_2>\h \eta ]\\
\to & \omega_f(\delta) \nu_\alpha (
(\eta-\delta),\infty)\nu_{\alpha_*} ( \eta,\infty)\\
\intertext{from \eqref{eq:obs1}}
\to & 0 \qquad (\delta \to 0).
\end{align*}
We can control $II$ by observing
\begin{align*}
II \leq & 2\|f\| t\Prob[V_1>\h \delta,  V_2>\h \eta]\\
\leq & 2\|f\| \frac{t}{ b_0^\leftarrow(  \h )          }           b_0^\leftarrow(  \h )  \Prob[V_1 \wedge V_2> b_0\circ b_0^\leftarrow(  \h ) \delta \wedge \eta]
\\
\to & 0 \qquad (t\to\infty),
\end{align*}
from \eqref{eq:2heavy}. The second term of \eqref{eq:notNu0} comes
from $B$ in a similar way to the derivation of $A$, relying on
\eqref{eq:obs2}. This completes case (1) where \eqref{eq:2heavy} holds.

For Case (2) when \eqref{eq:2light} holds, replace $\h$ with
$b_0(t)$ in \eqref{eq:pieces} and focus on $A$.  We compare with $f(\bV/b_0(t))$:
\begin{align*}
t\bfE\Bigl| f\Bigl(&\frac{\xi_1+V_1}{b_0(t)}, \frac{V_2}{b_0(t)}                 \Bigr) -
f\Bigl(\frac{\bV}{b_0(t)}\Bigr) \Bigr|\\
=&t\bfE|\cdot |1_{[\xi_1<b_0(t) \delta, V_1>b_0(\eta-\delta), V_2>b_0
  \eta ]}+
t\bfE|\cdot |1_{[\xi_1>b_0(t) \delta, V_2>b_0  \eta ]}= I+II.
\end{align*}
Since $t\bfE f(\bV/b_0(t)) \to \int f d\nu_0, $ we only have to show that
both $I$ and $II$ go to zero. For $I$ we have
\begin{align*}
I \leq & \omega_f(\delta) t\Prob[ V_1\wedge V_2 > b_o(t) (\eta
-\delta)\wedge \eta] \to \omega_f(\delta) ((\eta-\delta)\wedge \eta )^{-\alpha_0}\\
\to & 0 \qquad (\delta \to 0).
\end{align*}
Also using \eqref{eq:2light},
\begin{align*}
II \leq 2 \|f\| \frac{t}{b^\leftarrow (b_0)           b_*^\leftarrow(b_0)}
\Bigl(b^\leftarrow (b_0) \Prob[\xi_1>b\circ b^\leftarrow (b_0)
\delta]
b_*^\leftarrow(b_0) \Prob[V_2> b_*\circ b_*^\leftarrow (b_0) \eta
]\Bigr)\\
\sim (const) \frac{t}{b^\leftarrow (b_0)
  b_*^\leftarrow(b_0)} =(const) \frac{t}{h^\leftarrow (b_0)} \to 0.
\end{align*}
We can deal with the term $B$ similarly so this completes treatment of
Case (2).

Now consider Case (3) where \eqref{eq:justRight} holds. Again replace
$\h$ by $b_0(t)$ in \eqref{eq:pieces} and consider $A$. Write
\begin{align*}
2A=& t\bfE f\Bigl(\frac{\xi_1+V_1}{h(t)}, \frac{V_2}{h(t)}\Bigr)
\Bigl(1_{[\xi_1 \leq b_0(t) \delta]} + 1_{[\xi_1 > b_0(t) \delta]}
\Bigr)\\
=& t\bfE f\Bigl(\frac{\xi_1+V_1}{h(t)}, \frac{V_2}{h(t)}\Bigr)
-f\Bigl(\frac{\bV}{h(t)}\Bigr) 1_{[\xi_1 \leq b_0(t) \delta]} \\
&\qquad +  t\bfE f\Bigl(\frac{\xi_1+V_1}{h(t)}, \frac{V_2}{h(t)}\Bigr) 
-f\Bigl(\frac{\xi_1}{h(t)}, \frac{V_2}{h(t)}\Bigr) 1_{[\xi_1 > b_0(t)
  \delta]} \\
&\qquad + tEf(\bV/b_0(t))1_{[\xi_1 \leq b_0(t) \delta]}  
+t\bfE f\Bigl(\frac{\xi_1}{h(t)}, \frac{V_2}{h(t)}\Bigr) 1_{[\xi_1 > b_0(t)
  \delta]} \\
=& a+b+c+d.
\end{align*}
We have $c \to \int f(\bx) \nu_0 (d\bx)$ since $\Prob[\xi_1 \leq b_0(t)
\delta] \to 1.$ For $d$ note
$$d=t\bfE f\Bigl(\frac{\xi_1}{h(t)}, \frac{V_2}{h(t)}\Bigr) \to \int f
d\nu_\alpha \times \nu_{\alpha*}$$
using \eqref{eq:obs1} and the fact that \eqref{eq:justRight} is
equivalent to $h^\leftarrow (t)/b_0^\leftarrow (t) \to c^{-1}.$
Take the absolute value of $a$ and add to the indicator the event
$[V_1>b_0(t)(\eta -\delta)] $ (otherwise both terms in the difference
are zero) and 
\begin{align*}
|a|\leq & \omega_f(\delta) t\Prob[ V_1\wedge V_2 >b_0(t) (\eta -\delta)]\\
\to & \omega_f(\delta)  (\eta-\delta)^{-\alpha_0} \quad (t\to\infty)\\
\to & 0 \qquad (\delta \to 0).
\end{align*}
For $b$ write
\begin{align*}
|b| \leq & t\bfE \Bigl|f\Bigl(\frac{\xi_1+V_1}{h(t)}, \frac{V_2}{h(t)}\Bigr) 
-f\Bigl(\frac{\xi_1}{h(t)}, \frac{V_2}{h(t)}\Bigr)\Bigr| 1_{[\xi_1 > b_0(t)
  \delta, V_1\leq b_0(t) \delta]} \\
& \qquad +t\bfE |\cdot| 1_{[\xi_1 > b_0(t)
  \delta, V_1 >b_0(t) \delta]} =|b1|+|b2|.
\end{align*}
We dominate $|b1|$ by using the modulus of continuity
\begin{align*}
|b1|\leq & \omega_f(\delta) t\Prob[\xi_1 > b_0(t)
  \delta, V_2> b_0(t) \eta]
\end{align*}
where we added the condition on $V_2$ because otherwise, the probability would be
  zero due to the support of $f$ being bounded away from the axes. Let
  $t\to\infty$, apply \eqref{eq:obs1} and condition
  \eqref{eq:justRight} and then let $\delta \to 0$. Dominate $|b2|$ by
\begin{align*}
|b2| \leq &
 2\|f\| \Prob[\xi>b_0(t)\delta] t \Prob[V_1\wedge V_2 >b_0(t) \delta]\\
\sim & (const) \delta^{-\alpha_0}  \Prob[\xi>b_0(t)\delta]  \to 0 \quad (t\to\infty).
\end{align*}
The terms involving $B$ are handled similarly.
\end{proof}

\begin{exm}\label{eg:WhatAMess}
  \begin{figure}[h]
  \begin{centering}
 \includegraphics[width=10cm]{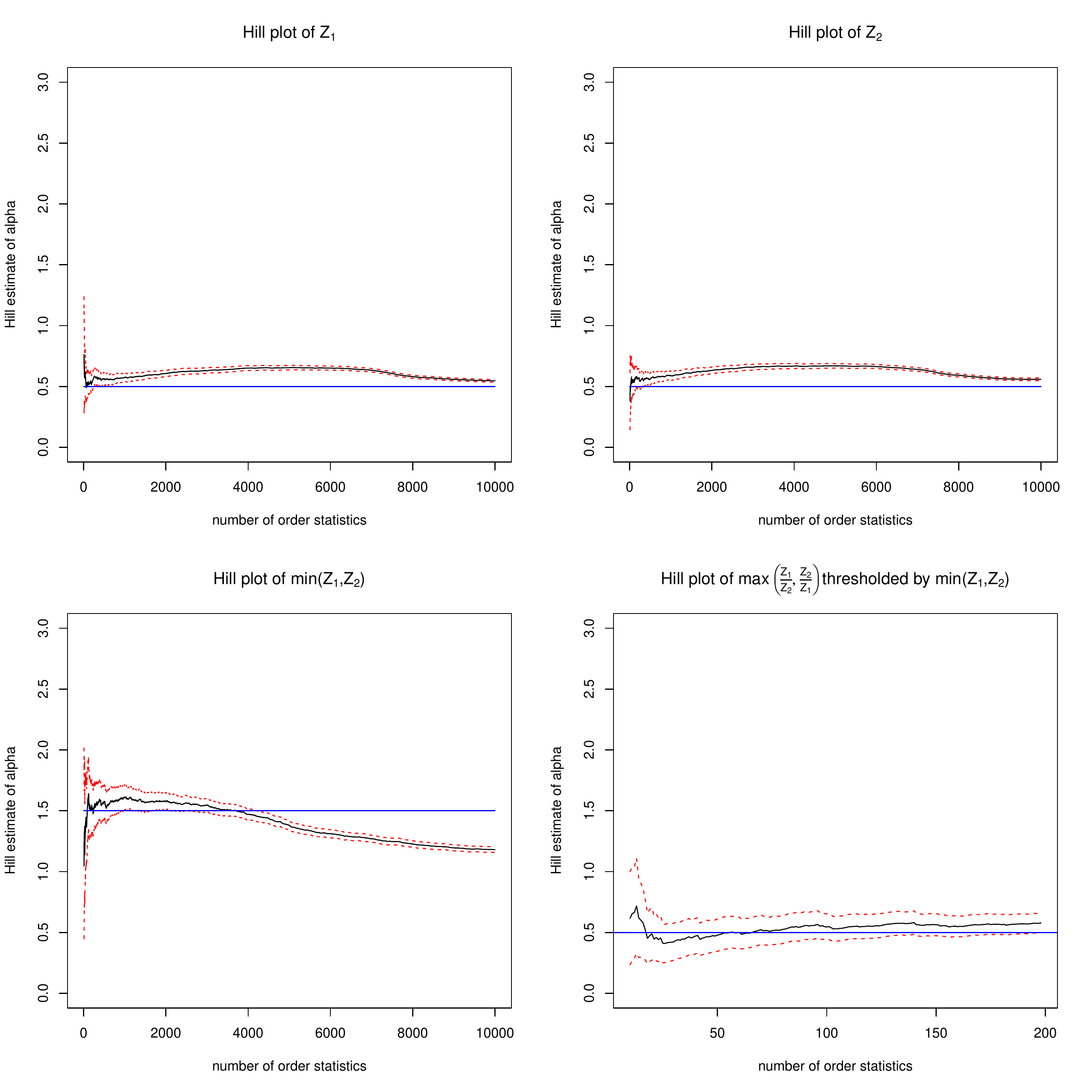}
  \end{centering}
  \caption{Exploratory plots for Example \ref{eg:WhatAMess}, Case 1,
    with $\alpha=0.5, \alpha^*=1, \alpha_0=2$. {\it Top panel:} Hill
    plots for the marginals $Z_1$ and $Z_2$. {\it Bottom left:} Hill
    plot for $\min\{Z_1,Z_2\}$. {\it Bottom right:} Hill plot for
    $\max\{Z_1/Z_2,Z_2/Z_1\}$ thresholded by the 200 largest values of
    $\min \{Z_1,Z_2\}$.}\label{fig:Ex321} 
  \end{figure} We illustrate instances of the three
  cases given in Proposition \ref{prop:3}. {We simulate data samples
    from three different regimes 
  as discussed in the Proposition \ref{prop:3} and estimate
  back the parameters of the additive model from which the data was
  generated. 
\begin{description}
\item[Case 1]     ${\alpha_*<\alpha_0 -\alpha}$.
 Let $\alpha=0.5,\, \alpha_* =1,\,
  \alpha_0=2$ and then 
$\alpha_*=1<1.5=\alpha_0-\alpha$. Let $\bY$ have the form
\eqref{eq:YhugAxes}
where $\xi_1,\xi_2$ are iid Pareto random variables with parameter $\alpha=0.5$. For $\bV$ it is
simplest to take $\bV=(V_1,V_2)$ iid Pareto $\alpha^*=1$ random
variables and hence we do so. Then $\alpha_0$ is 
the index of $V_1\wedge V_2$ and so $\alpha_0=2$. It is easy to see
that 
$\bZ= \bY + \bV \in \MRV(\alpha=0.5,t^{2}, \epsilon_{\{0\}} \times
\nu_{1/2} + \nu_{1/2}\times \epsilon_{\{0\}}, \E)$ with asymptotic
independence of the marginals. 

To verify that $\bZ \in
\MRV(\alpha+\alpha_*,t^{1/(\alpha+\alpha_*)},\nu_{\bZ, \text{hidden}},\E_0)=
\MRV(3/2, t^{2/3},\nu_{\bZ, \text{hidden}},\E_0)$ ab initio, take $\bz>\bzero$ 
and then
\begin{align*}
t\Prob[\bZ>t^{2/3}\bz]=&\frac t2 \Prob[\xi_1+V_1>t^{2/3}z_1,V_2>t^{2/3}z_2]\\
&\qquad+\frac t2 \Prob[V_1>t^{2/3}z_1,\xi_2+V_2>t^{2/3}z_2] =I+II.
\end{align*}
Focus on $I$ as treatment of $II$ is almost the same. We have
\begin{align*}
2I \sim & t ( t^{2/3}z_1)^{-1/2}
(t^{2/3}z_2)^{-1}=tt^{-2/3}t^{-1/3}z_1^{-1/2}z_2^{-1} \\
=&z_1^{-1/2}z_2^{-1},
\end{align*}
which is the first piece of the limit in \eqref{eq:notNu0}.

Hence we can check that the limit measure $\nu_{\bZ, \text{hidden}}$ in \eqref{eq:notNu0} has density
$$ \frac 14 z_1^{-3/2} z_2^{-2} + 
\frac 14 z_1^{-2} z_2^{-3/2} , \quad z_1>0, z_2>0$$
from which one can readily compute $G_1$ from \eqref{eq:G1} for $s>1$ as
$$\bar G_1(s) = \nu_{\bZ, \text{hidden}} \{\bz \in \E_0: 
z_1/z_2>s\}
=(const) s^{-1/2}.
$$
A similar calculation will lead to $G_2(s) = (const) s^{-1/2}, s>1$ meaning both $G_1$ and $G_2$ have regularly varying tail distributions with 
index $1/2$. In fact they are both Pareto (1/2) distributions.
 \begin{figure}[h]
  \begin{centering}
 \includegraphics[width=10cm]{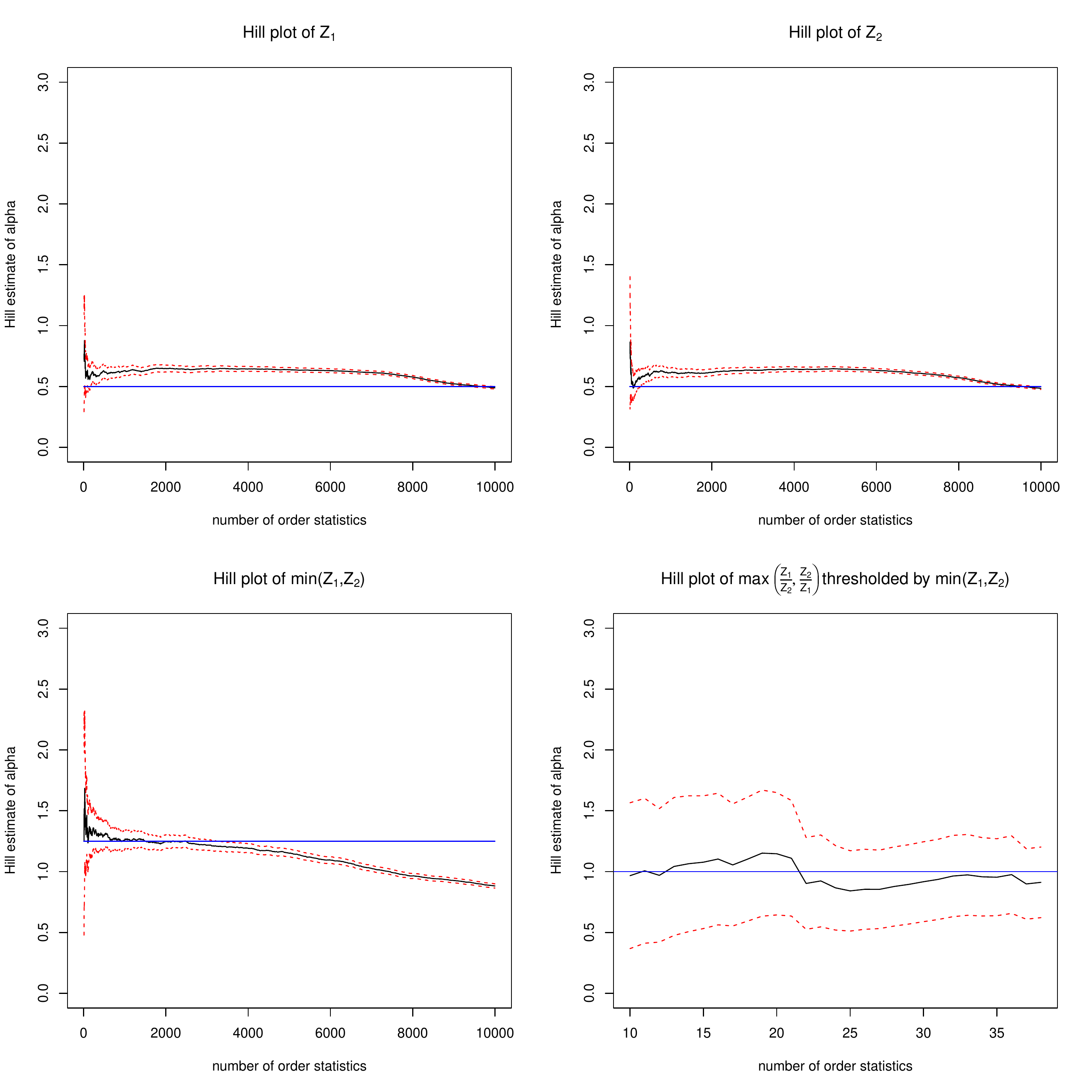}
  \end{centering}
  \caption{Exploratory plots for Example \ref{eg:WhatAMess}, Case 2,
    with $\alpha=0.5, \alpha^*=1, \alpha_0=1.25$. {\it Top panel:}
    Hill plots for the marginals $Z_1$ and $Z_2$. {\it Bottom left:}
    Hill plot for $\min\{Z_1,Z_2\}$. {\it Bottom right:} Hill plot for
    $\max\{Z_1/Z_2,Z_2/Z_1\}$ thresholded by the 200 largest values of
    $\min \{Z_1,Z_2\}$.}\label{fig:Ex322} 
  \end{figure}
We generate 10000 iid samples following the construction of $\bZ=\bY
+\bV$ described above and check whether we can estimate 
the regular variation index $\alpha=0.5$, the hidden regular variation index $\alpha+\alpha^*=1.5$ and the tail index of $G_1$ and $G_2$ from the sample.
Figure \ref{fig:Ex321} shows Hill plots for $Z_1$ and $Z_2$ in the top
panel, both of which indicate that the marginals  are heavy tailed
with parameter $\alpha=0.5$. 
The Hill plot of $\min\{Z_1,Z_2\}$ correctly identifies the HRV
parameter $\alpha+\alpha^*=1.5$. The final Hill plot of
$\max\{Z_1/Z_2,Z_2/Z_1\}$ for the 200 highest order statistics of
$\min \{Z_1,Z_2\}$ clearly indicates a  heavy tail with a tail index of $1/2$ for
both $G_1$ and $G_2$. Note since $G_1=G_2$, 
\eqref{eq:Gmix}
allows doing the estimation using the thresholded maxima of the component
ratios.

\item[Case 2]   $\alpha+\alpha_*>\alpha_0$.
 Let $\alpha=0.5,\, \alpha_* =1,\, \alpha_0=1.25$ and then
$\alpha_*=1 > 0.75 =\alpha_0-\alpha$. We generate $\bY$ in
exactly the same way as in Case 1. For $\bV$ we generate $R$, a Pareto
$\alpha_0=1.25$ random variable, $B$ a Bernoulli (1/2) random variable
and $\theta$ a Pareto $\alpha^*=1$ random variable. Now define: 
$$\bV= BR(\theta,1) + (1-B)R(1,\theta).$$ As in
 Case 1,   $\bZ = \bY + \bV \in
\MRV(\alpha=0.5,t^{2}, \epsilon_{\{0\}} \times
\nu_{1/2} + \nu_{1/2}\times \epsilon_{\{0\}} , \E)$ 
 and furthermore $\bZ = \bY + \bV \in 
\MRV(\alpha_0,t^{1/\alpha_0},\E_0)=
\MRV(1.25, t^{1/1.25},\E_0)$. Moreover by construction we have $G_1(s)
= G_2(s)=  s^{-1}, s>1$. Of course this is also clear from
Proposition \ref{prop:3}. 

We generate 10000 iid samples using the construction of $\bZ=\bY
+\bV$ and from this sample we estimate the regular variation
index $\alpha=0.5$, the hidden regular variation index $\alpha_0=1.25$
and the tail index of $G_1$ and $G_2$ which is $1$.   The top panels
in Figure
\ref{fig:Ex322}  display Hill plots for $Z_1$ and $Z_2$ that
indicate
the same tail
index of $\alpha=0.5$. The Hill plot for $\min\{Z_1,Z_2\}$
correctly indicates  a tail index of $\alpha_0=1.25$.  Finally, the Hill plot of
$\max\{Z_1/Z_2,Z_2/Z_1\}$ for the 200 highest order statistics of
$\min \{Z_1,Z_2\}$ indicates a tail index of $\alpha^*=1$ for both
$G_1\equiv G_2$.

 \begin{figure}[h]
  \begin{centering}
 \includegraphics[width=10cm]{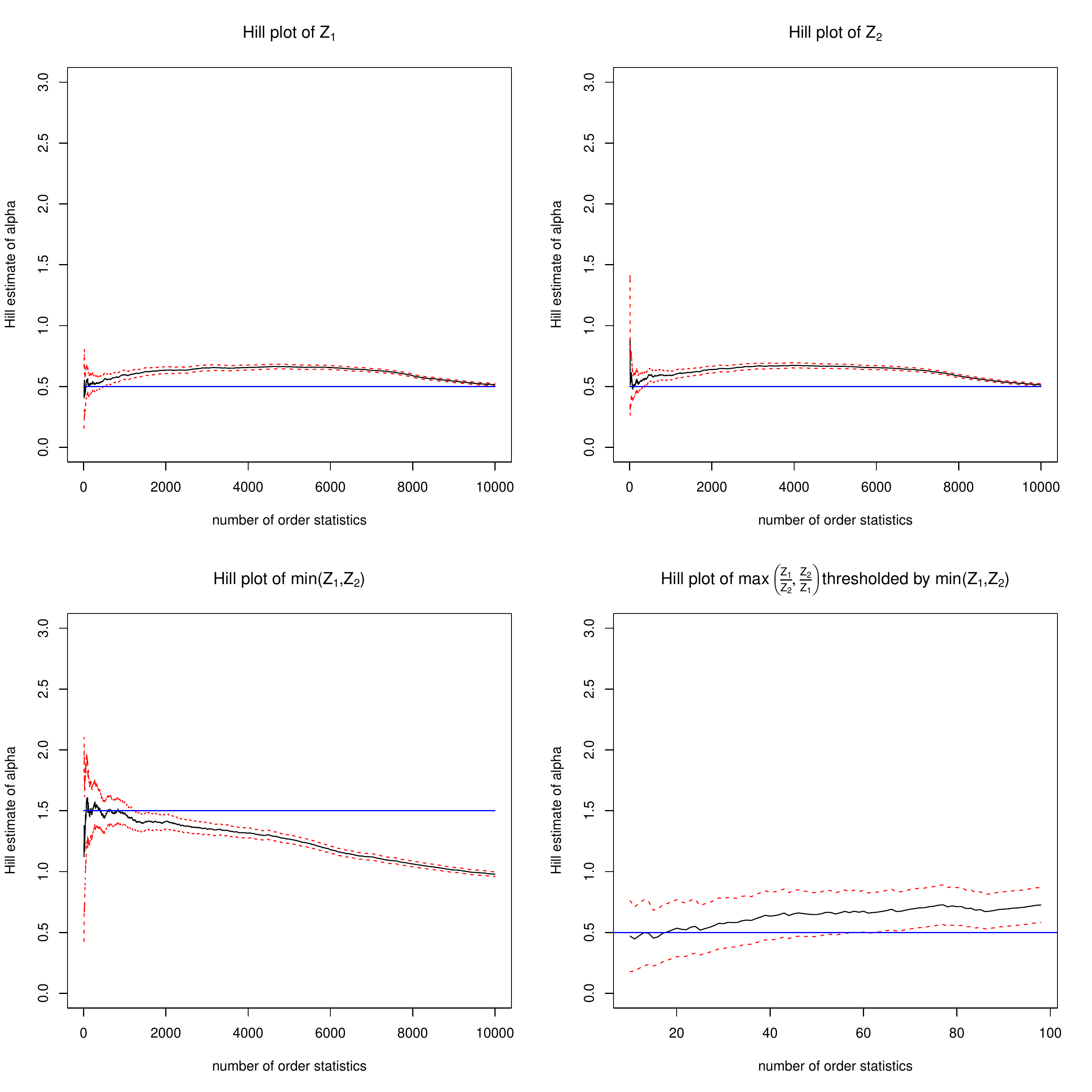}
  \end{centering}
  \caption{Exploratory plots for Example \ref{eg:WhatAMess}, Case 3,
    with $\alpha=0.5, \alpha^*=1, \alpha_0=1.5$. {\it Top panel:}
    Hill plots for the marginals $Z_1$ and $Z_2$. {\it Bottom left:}
    Hill plot for $\min\{Z_1,Z_2\}$. {\it Bottom right:} Hill plot for
    $\max\{Z_1/Z_2,Z_2/Z_1\}$ thresholded by the 200 largest values of
    $\min \{Z_1,Z_2\}$.}\label{fig:Ex323} 
  \end{figure}

\item[Case 3]  $\alpha+\alpha_*=\alpha_0$.
 Let $\alpha=0.5,\, \alpha_* =1,\, \alpha_0=1.5$ which satisfies
 $\alpha+\alpha_*=1.5=\alpha_0$. We generate $\bY$ as in
 Case 1 or 2 and generate $\bV$ using the method of Case 2,
 except that now $R$ is generated from a Pareto $\alpha_0=1.5$
 distribution.  
We verify that $\bZ = \bY + \bV \in \MRV(\alpha=0.5,t^{2}, \epsilon_{\{0\}} \times
\nu_{1/2} + \nu_{1/2}\times \epsilon_{\{0\}}, \E)$
and  $\bZ = \bY + \bV \in
\MRV(1.5,t^{1/1.5},  \nu_{\bZ}  ,\E_0)$. Getting the distribution of
$G_1$ and $G_2$ is more difficult in this case since the hidden limit
measure for $\bZ$ is more complicated as can be seen in
\eqref{eq:mixallnus}. A careful calculation shows that $G_1$ and $G_2$
have regularly varying tails with index $0.5$.   

We generate 10000 iid samples of $\bZ=\bY+\bV$
using this model .  
In Figure \ref{fig:Ex323}  the Hill plots for
$Z_1$ and $Z_2$ are in the neighborhood of 
$\alpha=0.5$ and the Hill plot for
$\min\{Z_1,Z_2\}$ 
correctly  indicates
 a tail index of $\alpha_0=1.5$ 
The  Hill plot of  $\max\{Z_1/Z_2,Z_2/Z_1\}$ for the 200 highest order
statistics of $\min \{Z_1,Z_2\}$ indicates a tail index of
$\alpha^*=0.5$ for both $G_1\equiv G_2$ which was what we were
expecting. 
\end{description}
}
\qed \end{exm}

\section{Detection and estimation: regular variation and hidden regular variation}\label{sec:detect}

What diagnostic tools exist to help us verify that multivariate data
come from a distribution possessing regular variation on some
domain?  Since regular variation is only an asymptotic tail property,
the task of deciding to use a multivariate regularly varying model is
challenging. 

Suppose we have $\bZ = (Z_1, Z_2)$ multivariate regularly varying on
$\E=[0,\infty)^2\setminus \{\bzero\}$. Under the transformation
$\gpolar$ as defined in \eqref{eq:defgpolar},
 $\|\bZ\|$ is regularly varying with some tail index $\alpha$ and
 \eqref{eq:limMeasPolar} holds. Diagnostics that investigate if $\bZ$ is regularly
 varying  often reduce the data to one dimension for instance
by taking norms or  max-linear combinations of $\bZ$ \citep[Chapter
 8]{resnickbook:2007} and then {apply} one dimensional heavy tail
  diagnostics  such as Hill or
 QQ plotting. We
 propose further diagnostics for the viability of a
 multivariate regularly varying model using the $\gpolar$
 transformation since $\gpolar $
 converts a
 regularly varying model to a {\it conditional extreme value\/}  (CEV) model
for which detection techniques exist \citep{das:resnick:2011b}.

\subsection{Detecting multivariate regular variation using the CEV model}\label{subset:detect:RVandCEV}
The {\it conditional extreme value model}
\cite{heffernan:resnick:2007,das:resnick:2011,das:resnick:2011b}
requires at least one of the marginals of the distribution  be  in the domain of
attraction of an extreme value distribution. In this section we
discuss  a modified version of the CEV model for bivariate random
vectors in the non-negative orthant where convergences are described
according to the notion of $\M$-convergence
\citep{lindskog:resnick:roy:2013,das:mitra:resnick:2013}. Define
$$\E_{\sqsupset}:=(0,\infty)\times[0,\infty) = [0,\infty)^2\setminus
\left([0,\infty)\times\{0\}\right).$$ 

\begin{dfn}\label{CEV_defn}
 {\rm{
 Suppose $(\xi,\eta) \in \R_{+}^2$ is a random vector and there exist functions   $a(t) \to \infty,\, b(t)>0$ for $t> 0$ and a non-null 
 measure $\mu \in \M(\E_{\sqsupset})$ such that in 
\begin{align}\label{eqn:CEV}
 & t P\left[\left(\frac{\xi}{{a}(t)},\frac{\eta}{{b}(t)}\right) \in \;
   \cdot\;\right] \to \mu(\cdot),\qquad \text{ in }\M (\E_{\sqsupset}).
\end{align}
Additionally assume that 
\begin{itemize}
 \item[(a)] $\mu((r,\infty] \times[0,s])$ is a non-degenerate measure in $s\in[0,\infty)$ for any fixed $r>0$, and,
 \item[(b)] $H(s):=\mu((1,\infty\times [0,s]))$ is a probability distribution. 
\end{itemize}
Then we say
 $(\xi,\eta)$ satisfies a conditional extreme value model and write $(\xi,\eta) \in \CEV({a,b},\mu)$.
}}
 \end{dfn}
 
 \begin{rem}
The definition has some  consequences \cite[Section 2]{heffernan:resnick:2007}:
\begin{enumerate}
\item Convergence in \eqref{eqn:CEV} implies that $\xi$ is regularly varying with some tail index $\alpha > 0$. Consequently $a(t)\in RV_{1/\alpha}$.
\item The limit $\mu$ is a product measure of the form \[\mu((r,\infty)\times[0,s]) = r^{-\alpha}H(s)=: \nu_{\alpha}(r,\infty) H(s) \] for all $(r,s)\in \E_{\sqsupset}$ if and only if 
$$\lim_{t\to \infty}\frac{b(tc)}{b(t)} =1.$$ 
\item If $a(t)=b(t), t> 0$ then $(\xi,\eta)$ is multivariate regularly varying on $\E_{\sqsupset}$ with limit measure $\mu$. (In such a case $\mu$ cannot be a product measure).
\end{enumerate}
 \end{rem}
 
 \begin{rem}
Statistical plots  that check whether bivariate data can be
  modelled by a CEV model were derived in \cite{das:resnick:2011}  and
  are based on the Hillish, Pickandsish and
Kendall's Tau statistics.
 If  data is generated from
a CEV model,  these statistics
tend to a constant as the sample size increases. We  concentrate
on the Hillish  and Pickandsish statistics for this
paper. We will further specialize to the case where $\mu$ is a product
measure $\mu = \nu_{\alpha} \times H$ for reasons that
will be clear in the next subsection.
 \end{rem}
 
 Suppose $(\xi_i,\eta_i); 1\le i \le n$ are iid samples in $\R^2_{+}$ and 
$(\xi_1, \eta_1) \in \CEV(a,b,\mu)$ for some
 $a(t)\to\infty, b(t)>0$ and $\mu\in \M(\E_{\sqsupset})$. We use the
 following notation: 

$$
\begin{array}{llll}
\xi_{(1)} \ge \ldots \ge \xi_{(n)} & \text{The decreasing order
statistics of  $\xi_1,\ldots,\xi_n$.}\\[2mm]
\eta_i^*, ~ 1 \le i \le n & \text{The $\eta$-variable corresponding to
$\xi_{(i)}$, also called}\\[2mm] 
& \text{ the concomitant of $\xi_{(i)}$.}\\ [2mm]
N_{i}^k= \sum\limits_{l=i}^k \bone_{\{\eta_l^* \le \eta_i^*\}}& \text{Rank of $\eta^*_i$ among
$\eta_1^*,\ldots,\eta_k^{*}$. We write $N_i=N_i^k$.}\\ [2mm]
\eta_{1:k}^* \le \eta_{2:k}^* \le \ldots \leq \eta_{k:k}^* & \text{The
increasing order statistics of $\eta_1^*,\ldots,\eta_k^*$.}\\[3mm]
\end{array}
$$

\noindent{\bf{Hillish statistic.}}
For $1\le k\le n$, the {\it Hillish statistic} is
 \begin{align}\label{def:bihill}
 \Hillish_{k,n}=\Hillish_{k,n}(\bbxi,\bbeta) := \frac{1}{k} \sum\limits_{j=1}^{k} \log \frac{k}{j} \log \frac{k}{N_{j}^k}
  \end{align}

 \begin{prop}[Proposition 2.2 and Proposition 2.3
   \cite{das:resnick:2011b}]\label{prop:convhillish} 
 Suppose $(\xi_i,\eta_i);$ $1\le i \le n$  are iid observations
from the $\CEV(a,b,\mu)$ model as  in Definition \ref{CEV_defn} and suppose $H$ 
is continuous. If $k=k(n) \to \infty, \, n
\to \infty$ and $k/n\to 0$, then
\begin{align} \label{eqn:hilllim}
 \Hillish_{k,n} \cinP \int\limits_{1}^{\infty}
\int\limits_{1}^{\infty}\mu((r^{\frac{1}{\alpha}},\infty)\times[0,H^{\leftarrow}(s^{-1})]) \frac{dr}{r}\frac{ds}{s} = : I_{\mu}.
\end{align} 
Moreover $\mu$ is a product measure if and only if both \begin{align*}
\Hillish_{k,n}(\bbxi,\bbeta) \cinP 1 \quad\text{and} \quad \Hillish_{k,n}(\bbxi,-\bbeta) \cinP 1. \end{align*}
\end{prop}

\noindent {\it Proof.} 
The proof follows from Propositions 2.2 and 2.3 in \cite{das:resnick:2011b}. The only difference here is the use of measure $\mu$ instead of $\mu^*$ and the roles of the first and the second components are switched.\\

\noindent{\bf{Pickandsish statistic.}}
This statistic gives another way to check the suitability of the CEV assumption and to detect a product measure
in the limit. The Pickandsish statistic is based on ratios of
differences of ordered concomitants and is patterned on the Pickands
estimate 
for the scale parameter of an extreme value distribution
(\citet{pickands:1975},
\citet[page 83]{dehaan:ferreira:2006},
\citet[page 93]{resnickbook:2007}). 
For notational convenience for $s\le t$ write
$\eta^*_{s:t}:=\eta^*_{\lceil s \rceil:\lceil t \rceil}$.  We define
the 
 Pickandsish statistic  for $0<q<1$ as
 \begin{align}\label{def:pick}
 \Pick_{k,n}(q) := \frac{\eta_{qk:k}^*-\eta_{qk/2:k/2}^*}{\eta_{qk:k}^*-\eta_{qk/2:k}^*}.
 \end{align}

 \begin{prop} [Proposition 2.4 and Corollary 2.5
   \cite{das:resnick:2011b}] \label{prop:convpick} 
 Suppose $(\xi_i,\eta_i);$ $1\le i \le n$  are iid observations
from the $\CEV(a,b,\mu)$ model as  in Definition
\ref{CEV_defn}. Assume that $k=k(n) \to \infty, \, n 
\to \infty$ and $k/n\to 0$. Then
 \begin{align}
 \Pick_{k,n}(q) & \cinP \frac{H\inv(q)(1-2^{\rho})}{H\inv(q)-H\inv(q/2)}, \label{eqn:rplim}
 \end{align}
provided $H\inv(q)-H\inv(q/2) \neq 0$. Here $\rho = (\log (c))^{-1} \log\left(\lim\limits_{t\to\infty}\frac{b(tc)}{b(t)}\right)$.
Moreover, $\mu$ is a product measure if and
  only if $$\Pick_{k,n}(q) \cinP 0$$ for some $0<q<1$ where $H\inv(q)-H\inv(q/2) \neq 0$.
  \end{prop}
 \noindent {\it Proof.} 
The proof follows from Proposition 2.4 in \cite{das:resnick:2011b}. The second part is immediate from \eqref{eqn:rplim}.

 \subsection{Relating MRV and CEV}
We have methods to detect a $\CEV$ model and indicate when the limit
is a product measure. What is the connection with multivariate regular
variation? This connection is given in
\eqref{eq:regVarE}--\eqref{eq:regVarE0Polar}. Regular variation of a
vector $\bZ$ on $\E$ 
and $\E_0$ with scaling functions $b(t)\in RV_{1/\alpha}$ and
$b_0(t)\in RV_{1/\alpha_0}$ respectively with $0<\alpha\le \alpha_0$
is equivalent to:
 \begin{align}
& tP\bigl[\bigl(\|\bZ \|/ b(t), \bZ/\|\bZ\|\bigr) \in \cdot \,\bigr]\to
 \nu_\alpha \times S(\cdot), \quad \text{ in } \M((0,\infty)\times
 \aleph_{\bzero }) \label{eq:regVarPolarE2nd}
\intertext{and}
& tP\Bigl[ 
 \Bigl(\frac{Z_1\wedge Z_2 }{b_0(t)} , \frac{\bZ}{Z_1\wedge Z_2} \Bigr) \in \cdot \Bigr]
 \to 
 \nu_{\alpha_0} \times S_0 (\cdot) \qquad \text{ in }
 \M\bigl((0,\infty)\times \aleph_{[\axes]}).  \label{eq:regVarPolarE02nd}
 \end{align}
If $\aleph_{\bzero}$ and $\aleph_{[\axes]}$ were
 subsets of $[0,\infty)$ we could conclude that
 \eqref{eq:regVarPolarE2nd} and \eqref{eq:regVarPolarE02nd} describe
CEV  models and modest changes, described in the next two results,
  allow use of the CEV model diagnostics.
 
 \begin{prop}\label{prop:E} Suppose $\bZ$ is a random element of
   $\mathbb{R}_+^2. $  
  Fix a norm for $\bz\in \R_{+}^2: \|(z_1,z_2)\|= z_1+z_2$. 
Then $\bZ
  \in \MRV(\alpha, b(t),\nu, \E)$ (which means \eqref{eq:regVarEPolar}
also holds)  if and only if  $\left(\|Z\|,\frac{Z_1}{\|Z\|}\right) \in
CEV(b,1,\mu )$ with limit measure $\mu= \nu_{\alpha} \times \bar{S} $
where $\bar{S}(A) = S((x,y)\in \aleph_{\bzero }: x\in A)$ for any
$A\in\mathcal{B}[0,\infty)$. 
  \end{prop}
 \noindent{{\it Proof.}} The proof is easily deducible from the relationship between $S$ and $\bar{S}$.

 \begin{prop}\label{prop:E0} 
  Suppose $\bZ \ge 0$ is regularly varying on $\E$ with function $b(t)\in RV_{1/\alpha}$. Then $\bZ$ exhibits HRV on $\E_0$ with scaling function $b_0(t)\in RV_{1/\alpha_0}, \alpha_0\ge \alpha$  if and only if  
$$\left( Z_1\wedge Z_2, \left(\frac{Z_1}{Z_2} \bigvee
    \frac{Z_2}{Z_1}\right)\right) \in \CEV(b_0,1,\mu_0  )$$ 
with limit measure
given by $\mu_0= \nu_{\alpha_0} \times \left(pG_1+(1-p)G_2\right)  $  
  where $G_1(s)=S_0([1,s]\times \{1\})$ and
  $G_2(s)=S_0(\{1\}\times[1,s])$ for $s\ge 1$ and $G_1(s)=G_2(s)=0,
  s\le 1.$ 
 \end{prop}
\begin{proof}
  The proof follows from the connection
   between $S_0$ and $G_1,G_2$.  \end{proof}
  
\section{Testing for MRV and HRV: data examples}\label{sec:detect:data}
Here we analyze data sets to see whether a multivariate regularly varying model is a valid assumption. We also look for asymptotic independence and if it exists we test for the existence of hidden regular variation.

\begin{exm}\label{ex:1}
\noindent{\bf{Boston University: HTTP downloads.}}
\begin{figure}[h]
\begin{centering}
\includegraphics[width=9cm]{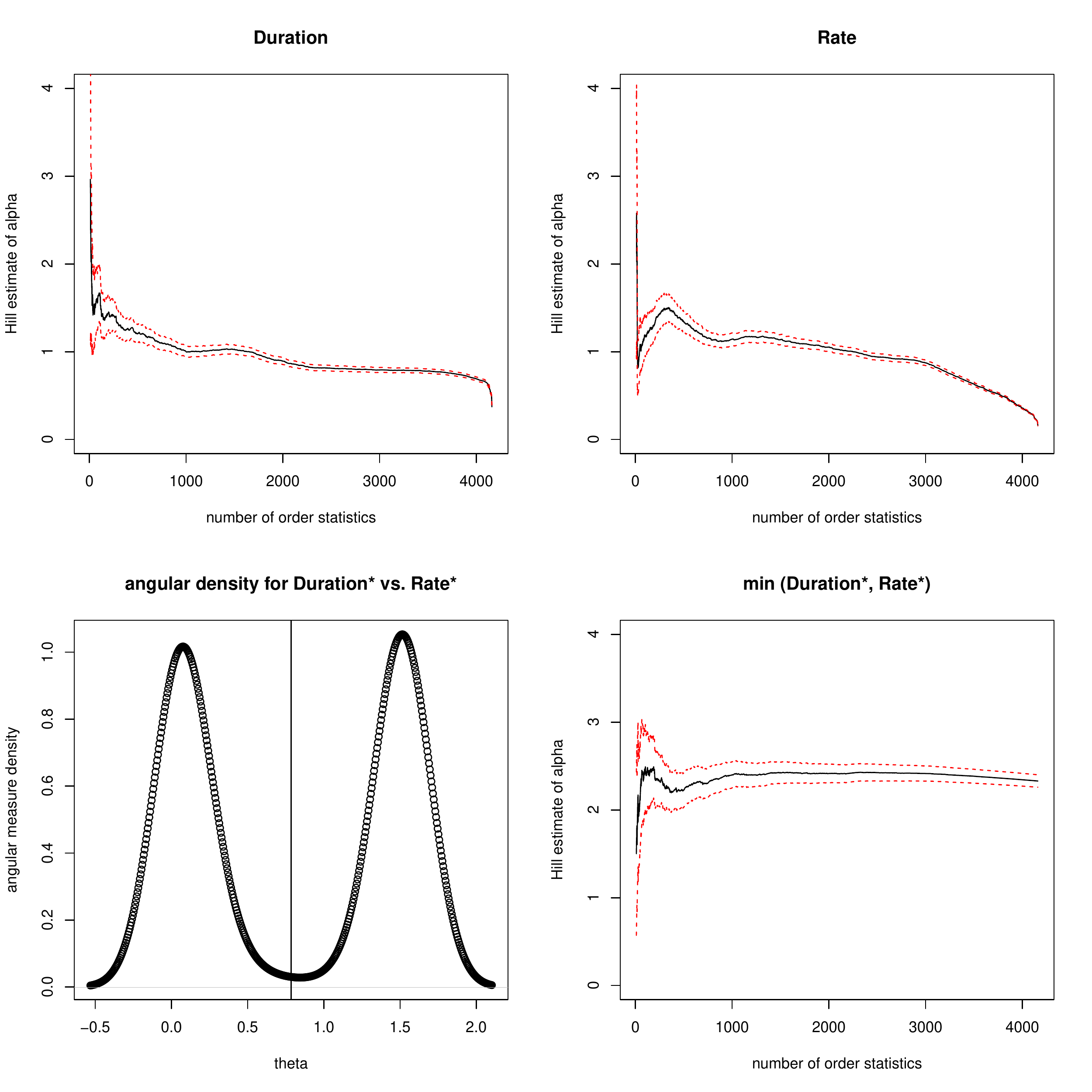}
\end{centering}
\caption{BU dataset. {\it Top panel:} Hill plots of tail parameters for $D$ and $R$. {\it Bottom left plot:} angular density of $(D^*,R^*)$. {\it Bottom right plot:} Hill plot for $\min(D^*,R^*)$.}\label{fig:buLbuRalpha}
\end{figure}

The first data set is obtained from a now classical Boston University
study \citep{crovella:bestavros:1996} 
which suggested self-similarity and heavy tails in web-traffic
data. Our dataset was created from  HTTP downloads in sessions
initiated by logins at a Boston University computer laboratory. It
consists of 8 hours 20 minutes worth of downloads in February 1995
after applying an aggregation rule to downloads to associate machine
triggered actions with human requests and is discussed in \cite[page
176]{guerin:nyberg:perrin:resnick:rootzen:starica:2003}. The result of
the aggregation 
is {4161} downloads which are characterized by the following variables:
\begin{itemize}
\item $S =$ the size of the download in kilobytes, 
\item $D =$ the duration of the download in seconds,
\item $R =$ throughput of the download; that is,  $= S/D$.
\end{itemize}

Previous studies  \citep[page 299,
316]{resnickbook:2007} indicate heavy tailed behavior of all three
variables and asymptotic independence between 
$D$ and $R$. We concentrate on the variables $(D,R)$
so our data is $\{(D_i,R_i); 1\le i \le
4161\}$. Moreover the rank-transformed variables are denoted:  
    \begin{align*}
D_i^*&= \sum_{j=1}^{4161} \bone_{\{D_i\ge D_j\}},
&& R_i^*= \sum_{j=1}^{4161} \bone_{\{R_i\ge R_j\}}.
\end{align*}
for $1\le i\le 4161$ with the generic rank-transformed variables denoted $D^*$ and $R^*$ respectively.

\begin{figure}[h]
\begin{centering}
\includegraphics[width=13cm]{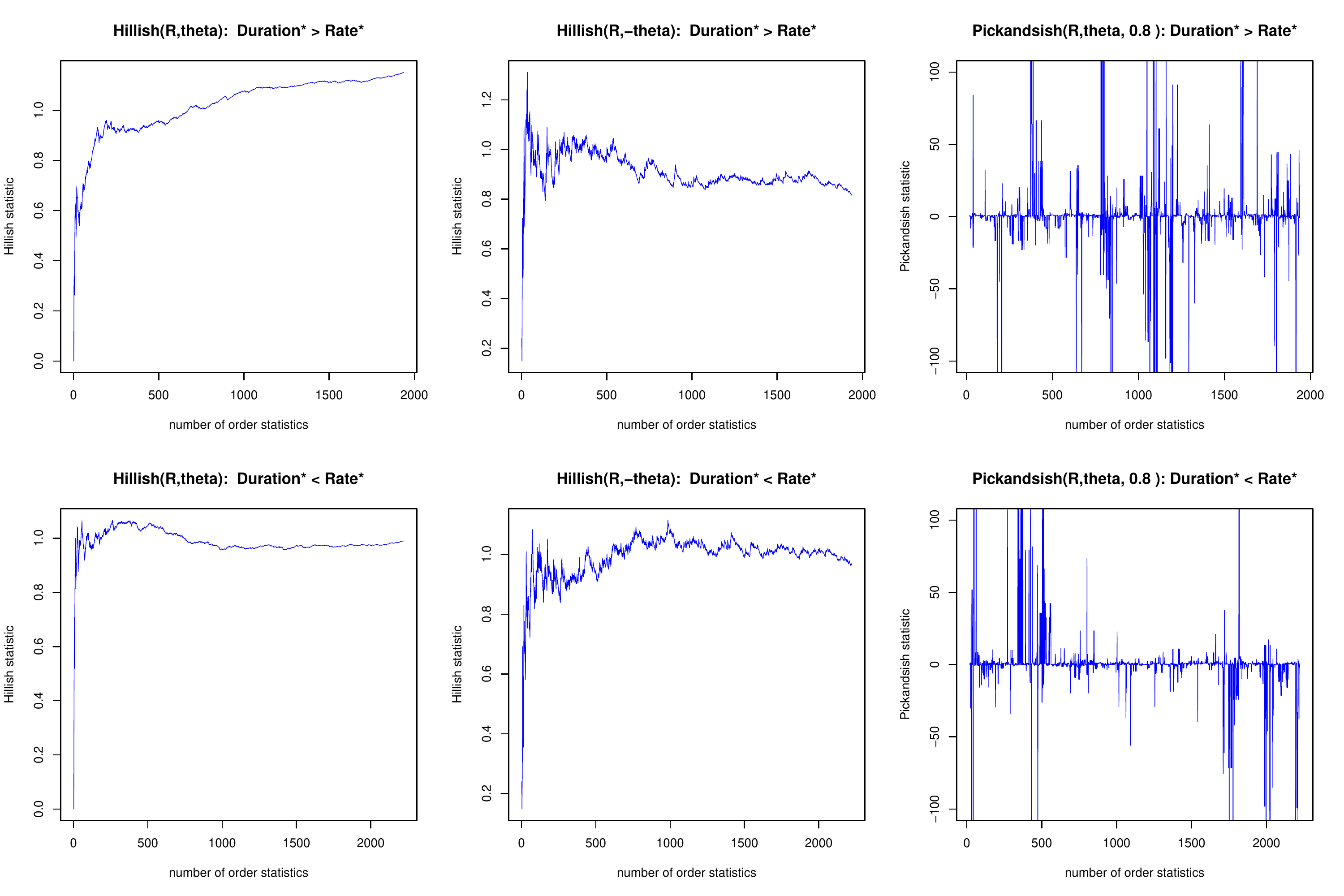}
\end{centering}
\caption{BU dataset. {\it Top panel ($D^*>R^*$):} Hillish plots for
  $(A,\theta_1)$ and $(A,-\theta_1)$ and Pickandsish plot for
  $(A,\theta_1)$ at $q=0.8$. {\it Bottom panel ($D^*<R^*$):} Hillish
  plots for $(A,\theta_2)$ and $(A,-\theta_2)$ and Pickandsish plot
  for $(A,\theta_2)$ at $q=0.8$.}\label{fig:buLbuRhillpick} 
\end{figure}

In Figure \ref{fig:buLbuRalpha} we plot Hill estimates of the tail
parameters of $D$ and $R$ for increasing number of order statistics of their
respective univariate data values. Both plots are consistent with
$D$ and $R$ being heavy tailed 
 with  tail parameters $\alpha_D$ and $\alpha_R$
greater than 1. (This is confirmed
  \citep{resnickbook:2007,drees:dehaan:resnick:2000,resnick:starica:1997a}
 by altHill and QQ plots--not
shown--showing $\hat \alpha_D=1.4$ and $\hat \alpha_R=1.2$.)
 The angular density plot of $(D^*,R^*)$ 
shows data concentration near $0$ and $\pi/2$ consistent with
asymptotic independence of the quantities. Asymptotic independence
does not automatically imply HRV so we check for HRV on $\E_0$. 

\begin{figure}[h]
\begin{centering}
\includegraphics[width=6cm]{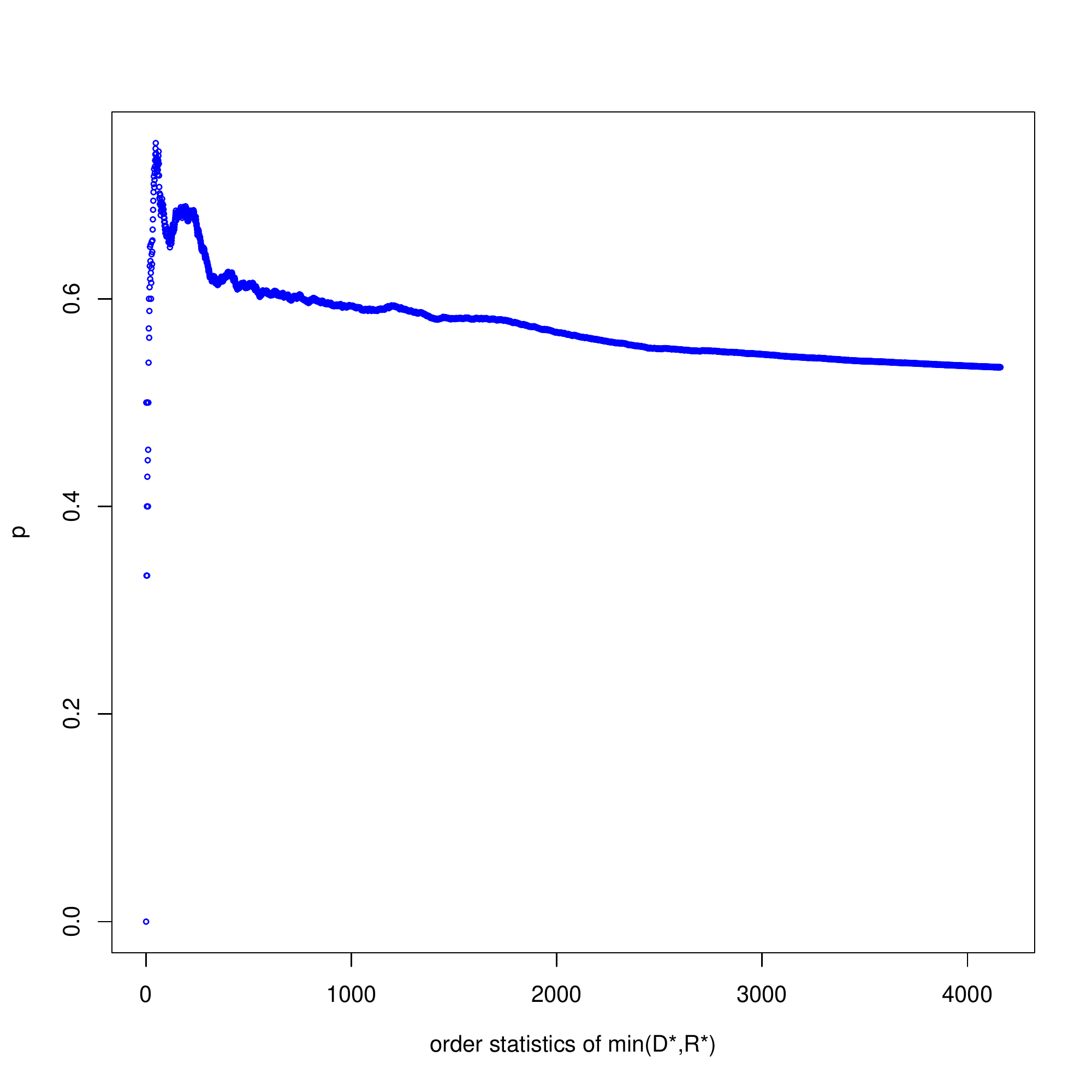}
\end{centering}
\caption{BU dataset. Proportion of data with $D_i^*>R_i^*$ for order
  statistics of $A_i=\min\{D_i^*,R_i^*\}$.}\label{fig:buLbuR_p} 
\end{figure}


\begin{figure}[h]
\begin{centering}
\includegraphics[width=12cm]{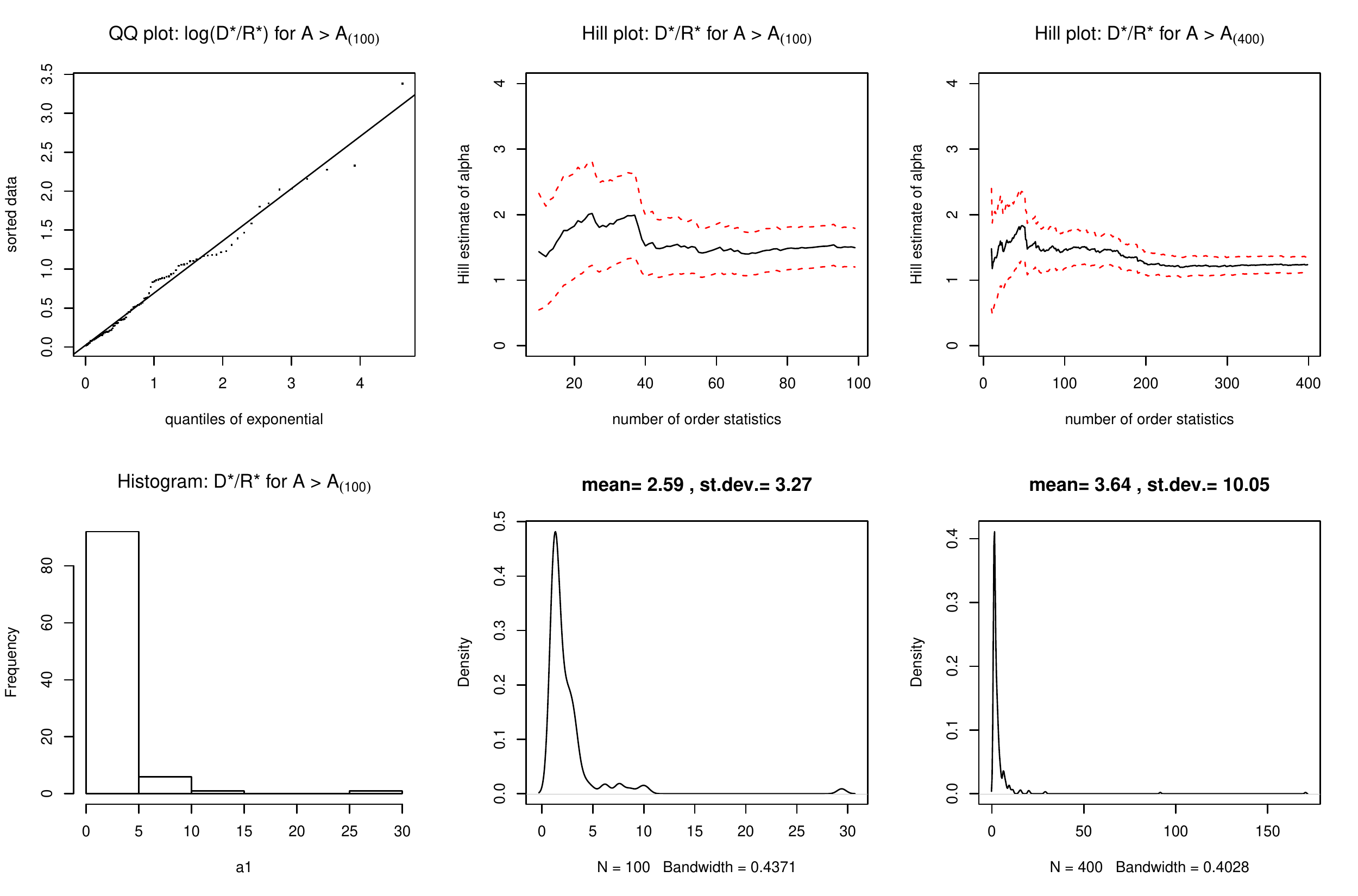}
\end{centering}
\caption{BU dataset. {\it Top panel:} QQ plot   
  of $\log(D^*/R^*)$ when $A_i>A_{(100)}$ and Hill plots of $D^*/R^*$
  when $A_i>A_{(100)}$ and $A_i>A_{(400)}$. 
{\it Bottom panel:} Histogram of $D^*/R^*$ when $A_i>A_{(100)}$ and
kernel density estimates of $D^*/R^*$ when $A_i>A_{(100)}$ and
$A_i>A_{(400)}$. }\label{fig:buLbuRG1} 
\end{figure}

\begin{figure}[h]
\begin{centering}
\includegraphics[width=12cm]{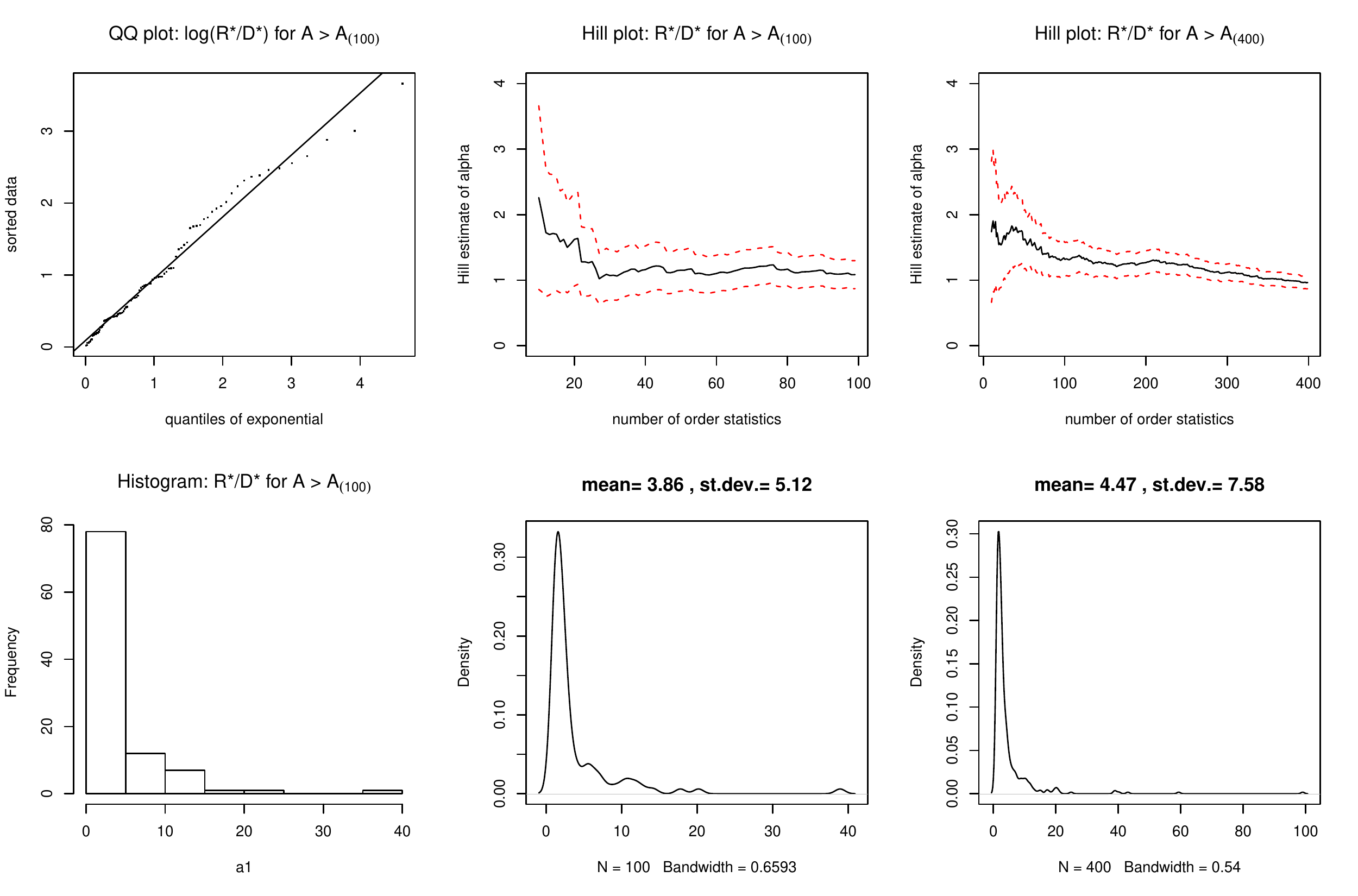}
\end{centering}
\caption{BU dataset. {\it Top panel:} QQ plot
of $\log(R^*/D^*)$ when $A_i>A_{(100)}$ along with Hill plots of $R^*/D^*$
when $A_i>A_{(100)}$ and $A_i>A_{(400)}$. 
{\it Bottom panel:} Histogram of $R^*/D^*$ when $A_i>A_{(100)}$ and kernel density estimates of  $R^*/D^*$ when $A_i>A_{(100)}$ and $A_i>A_{(400)}.$ }\label{fig:buLbuRG2}
\end{figure}
The rank transformation causes $(D^*,R^*)$ to be  standard
regularly varying with $\alpha=1$ and Proposition \ref{prop:E0}
implies 
$(D^*,R^*)$ has hidden regular variation on $\E_0$ if (and only if) $$
(A,\theta):=\left(\min\{D^*,R^*\}, \max \left\{\frac{D^*}{R^*},
    \frac{R^*}{D^*}\right\}\right) \in \CEV (b_0,1,\mu_0 )$$
for some function $b_0$. We proceed by testing the following:

\begin{enumerate}
\item Is the variable $A= \min\{D^*,R^*\}$  regularly varying with parameter greater
  than 1? The bottom right plot in Figure \ref{fig:buLbuRalpha} plots
  Hill estimates for increasing number of order statistics of $A$ and 
 stabilizes between 2 and 3 indicating the desired  heavy tail behavior. 
\item For $D^*>R^*$, we check whether $(A,\theta_1):= (\min\{D^*,R^*\}, 
  \frac{D^*}{R^*})$ follows a CEV model. In the top panel of Figure
  \ref{fig:buLbuRhillpick}, the Hillish plots of
  $(A,\theta_1)$ and $(A,-\theta_1)$ are close to 1 near the left side
  of their plots. Moreover we observe that the Pickandsish estimate at
  $q=0.8$ also remains near 0. From Propositions
  \ref{prop:convhillish} and \ref{prop:convpick}, this is consistent
with  $(A,\theta_1)\in \CEV (b_0,1 , \mu_0 )$ with a limit 
  measure of the CEV being a product measure. 

\item For $D^*<R^*$, we similarly check whether $(A,\theta_2):=
  (\min\{R^*,D^*\}, \frac{R^*}{D^*})$ follows a CEV model. In the bottom
  panel of Figure \ref{fig:buLbuRhillpick} we observe that the Hillish
  plots of $(A,\theta_2)$ and $(A,-\theta_2)$ are close to 1 near the
  left side of their plots. We also observe that the Pickandsish
  estimate at $q=0.8$  remains near 0. Hence we again conclude that
  the evidence is consistent with
  $(A,\theta_2)\in \CEV (b_0,1,\mu_0)$ with a limit measure of the CEV being
  a product measure. 
\end{enumerate}
\par

Thus modeling the joint distribution of $(D,R)$ using MRV and HRV
is consistent with the data.
 The next 
step is to  estimate the
 distributions of $\theta_1\sim G_1$ and $\theta_2\sim G_2$ as well as
 $q$ defined in Proposition~\ref{prop:E0}.  Figure
 \ref{fig:buLbuR_p}  plots  
$ \hat{q}_k = \frac{1}{k} \sum_{i=1}^{4161} \bone_{\{D_i^*>R_i^*, A_i
  > A_{(k)}\}}, \quad k=2,\ldots, 4161,$
where $A_i= \min\{R^*_i,D^*_i\}$ and $A_{(1)} \ge A_{(2)} \ge \ldots$ form
order statistics from $A_i; 1\le i \le 4161$. Observing Figure
\ref{fig:buLbuR_p} for $k$ near 0,  an estimate of $q$ is
$\hat{q}=0.6$.

To find the distribution of $\theta_1$ we
make a standard exponential QQ plot of $\log(D^*_i/R^*_i)$ where
$A_i=\min (D^*_i,R^*_i)> A_{(100)}$, which serves as an
  exploratory diagnostic for heavy tails. 
 We also create Hill plots for $D^*_i/R^*_i$ where $A_i >
   A_{(k)}$ for two choices of $k$.
The top panels of Figure \ref{fig:buLbuRG1} 
give the QQ plot for $k=100$ (left) and the Hill
plots for $k=100$ and $400$ (middle and right). The bottom
panels in Figure \ref{fig:buLbuRG1}  have a histogram of $D_i^*/R_i^*$
for $A_{i}>A_{(100)}$ (left) and kernel density plots of $D_i^*/R_i^*$
for $A_{i}>A_{(100)}$ (middle) and $A_{i}>A_{(100)}$ (right).  The
plots 
indicate $G_1$ is heavy tailed with an index between 1.5 and 
2 and we can provide a density estimate for the distribution of
$\theta_1$. 

 We create the same set of plots for finding $G_2$ in  Figure
\ref{fig:buLbuRG2} which also indicates towards a similar conclusion
of heavy tailed behavior for $G_2$ with an index close to but less
than $2$. 
\end{exm}

\begin{exm}\label{ex:2}
\noindent{\bf{UNC Chapel Hill HTTP response data.}}
\begin{figure}[h]
\begin{centering}
\includegraphics[width=12cm]{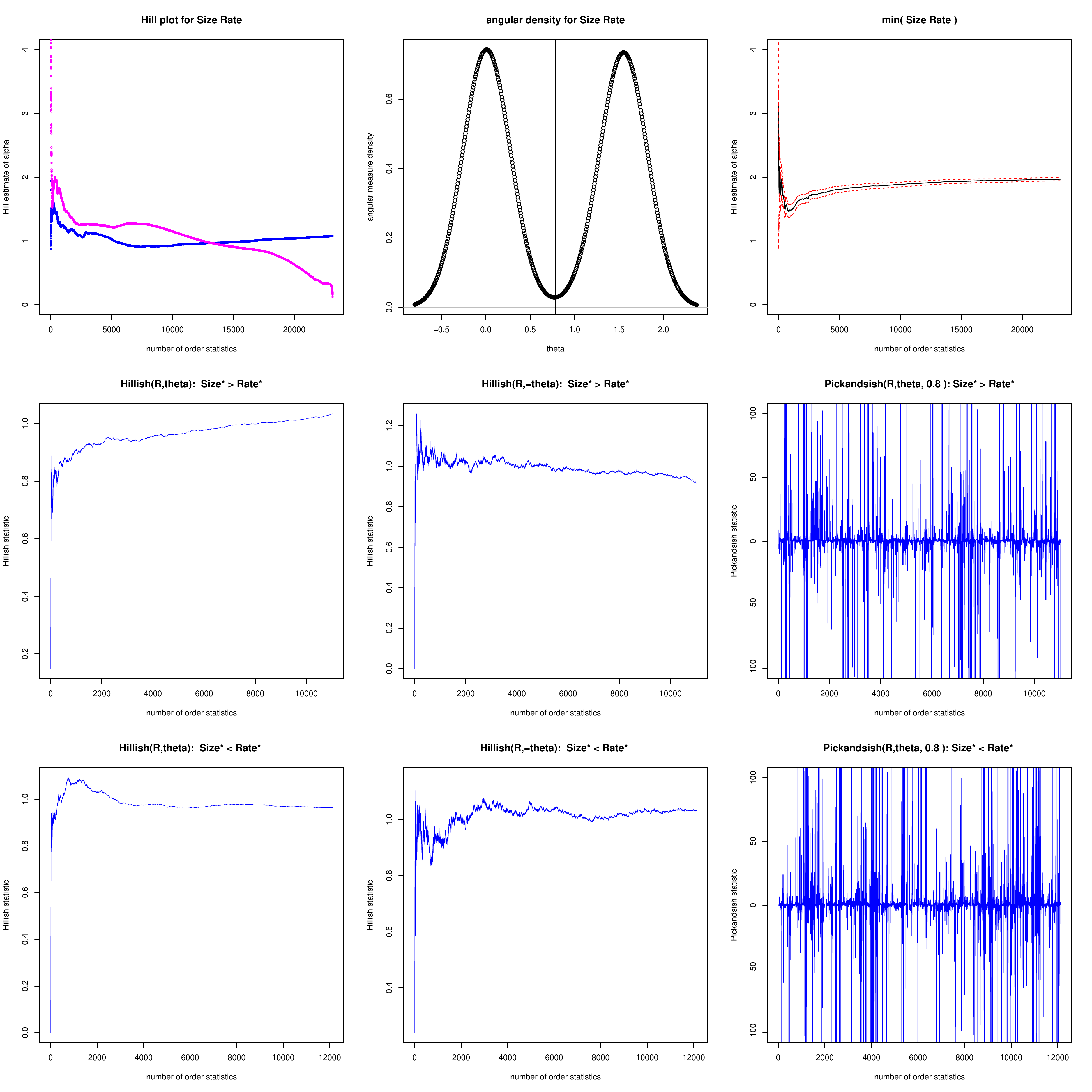}
\end{centering}
\caption{UNC HTTP responses dataset. {\it Top panel:} (Left:) Hill plots of tail parameters for $S$(blue), $R$(magenta); (Middle:) angular density of $(S^*,R^*)$; (Right:) Hill plot for $\min(S^*,R^*))$. {\it Middle panel ($S^*>R^*$):} Hillish plots for
  $(A,\theta_1)$ and $(A,-\theta_1)$ and Pickandsish plot for
  $(A,\theta_1)$ at $q=0.8$. {\it Bottom panel ($S^*<R^*$):} Hillish
  plots for $(A,\theta_2)$ and $(A,-\theta_2)$ and Pickandsish plot
  for $(A,\theta_2)$ at $q=0.8$.} \label{fig:UNCall}
\end{figure}
 A {\it response\/}
is the data transfer resulting from an HTTP request.
The  data set 
\citep{hernandezcamposetal:2005} consists of 
21,828 thresholded responses bigger  than  100 kilobytes 
measured between 1:00pm and 5:00pm on 25th April, 2001.
We use similar notation as in Example \ref{ex:1}.

\begin{itemize}
\item $S =$ HTTP response size; total size of packets transferred in kilobytes,
\item $D =$ the elapsed duration between first and last packets in
  seconds of the response,
\item $R =$ throughput of the response  $= S/D$.
\begin{figure}[h]
\begin{centering}
\includegraphics[width=4.5cm]{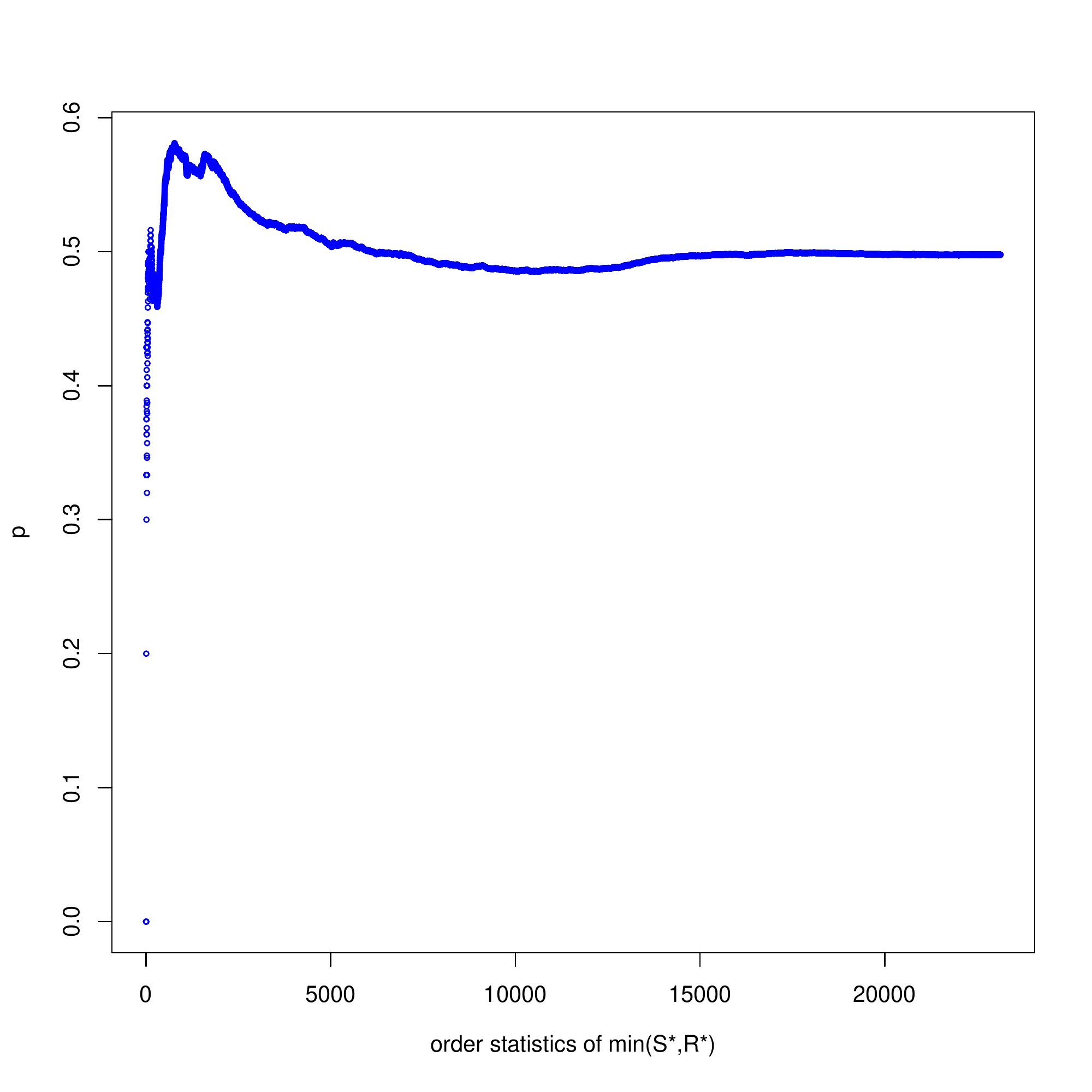}
\end{centering}
\caption{UNC HTTP responses. Proportion of data with $S_i^*>R_i^*$ for order
  statistics of $A_i=\min\{S_i^*,R_i^*\}$.}\label{fig:25Apr_p} 
\end{figure}
\end{itemize}

Thus, the data set consists of 
$\{(S_i,D_i,R_i); 1\le i \le 21828 \}$. Our interest is in
the variables $(S,R)$ which  exhibit heavy  tails and asymptotic independence 
\cite{hernandezcamposetal:2005}.  Denote the rank-transformed
variables:  
$$\Bigl(
S_i^*= \sum_{j=1}^{21828} \bone_{\{S_i\ge S_j\}},
R_i^*= \sum_{j=1}^{21828} \bone_{\{R_i\ge R_j\}}\Bigr), \quad 1\le i\le 21828,
$$
 with the generic rank-transformed variables denoted $S^*$ and $R^*$
 respectively. The top left plots in Figure \ref{fig:UNCall}  give
 Hill plots of the tail  indices of the distributions of $S$ and  $R$
 and suggest these indices are between 1 and 2.
Asymptotic independence of $S,R$ is exhibited in the angular density
plot (top middle plot) for $(S^*,R^*)$.


We next inquire if  HRV exists on $\E_0$.
The Hill plot for $\min(S^*,R^*)$ on the upper right panel of Figure
\ref{fig:UNCall}  gives a tail estimate $\hat \alpha_0$ clearly
greater than 1 and is consistent with HRV.
We transform the data $\{(S^*,R^*);\,1\leq i\leq 21828\}$
with the transformation $\gpolar$) to obtain:
 $$
(A,\theta):= \gpolar(S^*,R^*)=\left(\min\{S^*,R^*\}, \max \left\{\frac{S^*}{R^*},
    \frac{R^*}{S^*}\right\}\right).$$ 
From Proposition \ref{prop:E0} we know $(A,\theta)\in \CEV(b_0,1,\mu_0)$ for some function $b_0$ and measure $\mu_0$ on $\E_0$. 
 For both the cases $S^*>R^*$ (see middle panels in Figure
 \ref{fig:UNCall}) and $S^*<R^*$ (see bottom panels in Figure
 \ref{fig:UNCall}), we employ the Hillish and Pickandsish diagnostics
 to check consistency of  $(A,\theta_1):= (\min\{S^*,R^*\}, 
{S^*}/{R^*})$  and $(A,\theta_2):= (\min\{S^*,R^*\}, 
{R^*}/{S^*})$ with the CEV model with product limit measure.
    The Hillish plots are reasurringly hovering at height 1 and the
    Pickandsish plots center at 0.
\begin{figure}[h]
\begin{centering}
\includegraphics[width=9cm]{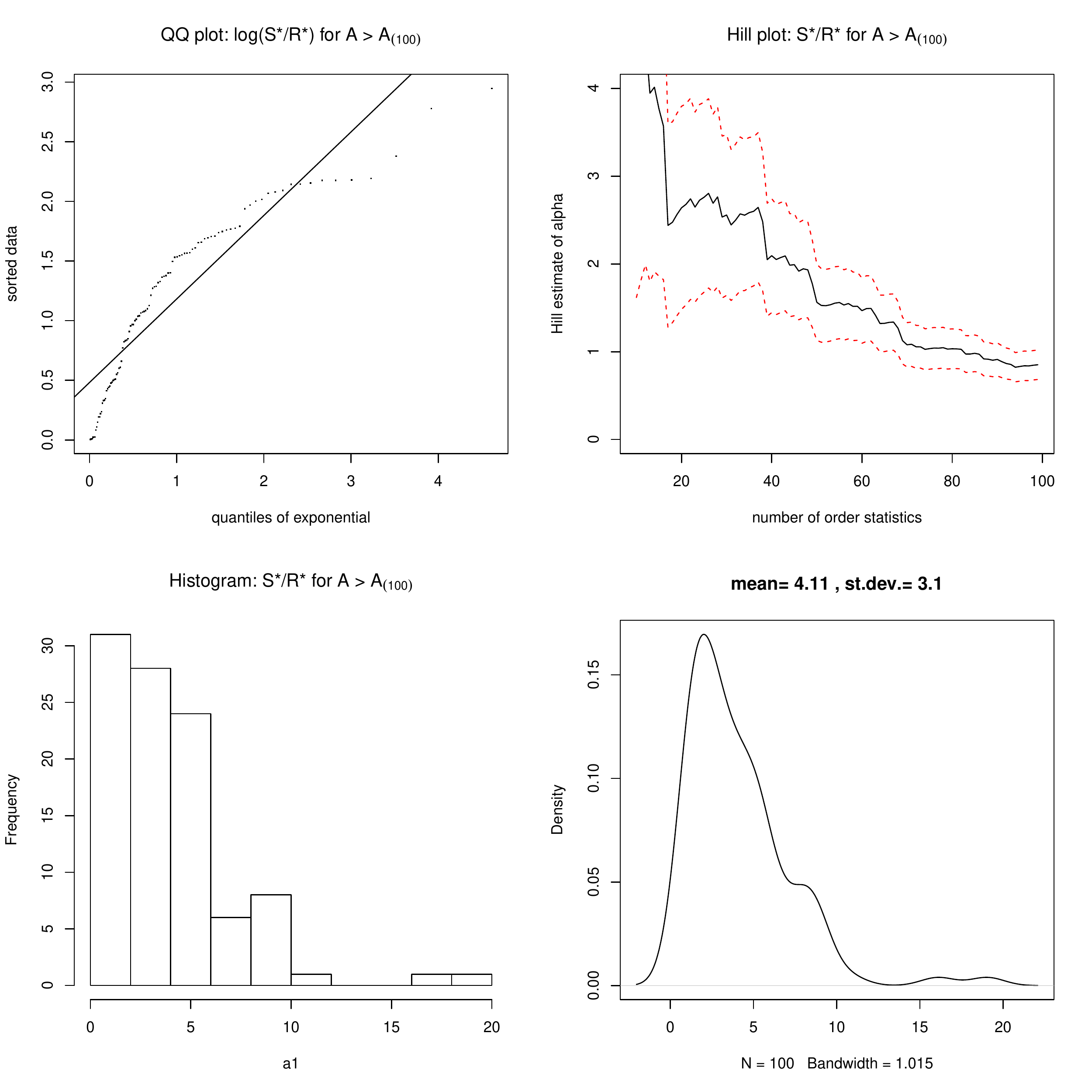}
\end{centering}
\caption{UNC HTTP responses: {\it Top:} QQ plot   
  of $\log(S^*/R^*)$ when $A_i>A_{(100)}$ along with Hill plots of $S^*/R^*$
  when $A_i>A_{(100)}$ and $A_i>A_{(400)}$. 
{\it Bottom:} Histogram and kernel density estimates of $S^*/R^*$ when $A_i>A_{(100)}$}\label{fig:25Apr_G1} 
\end{figure}

\begin{figure}[h]
\begin{centering}
\includegraphics[width=12cm]{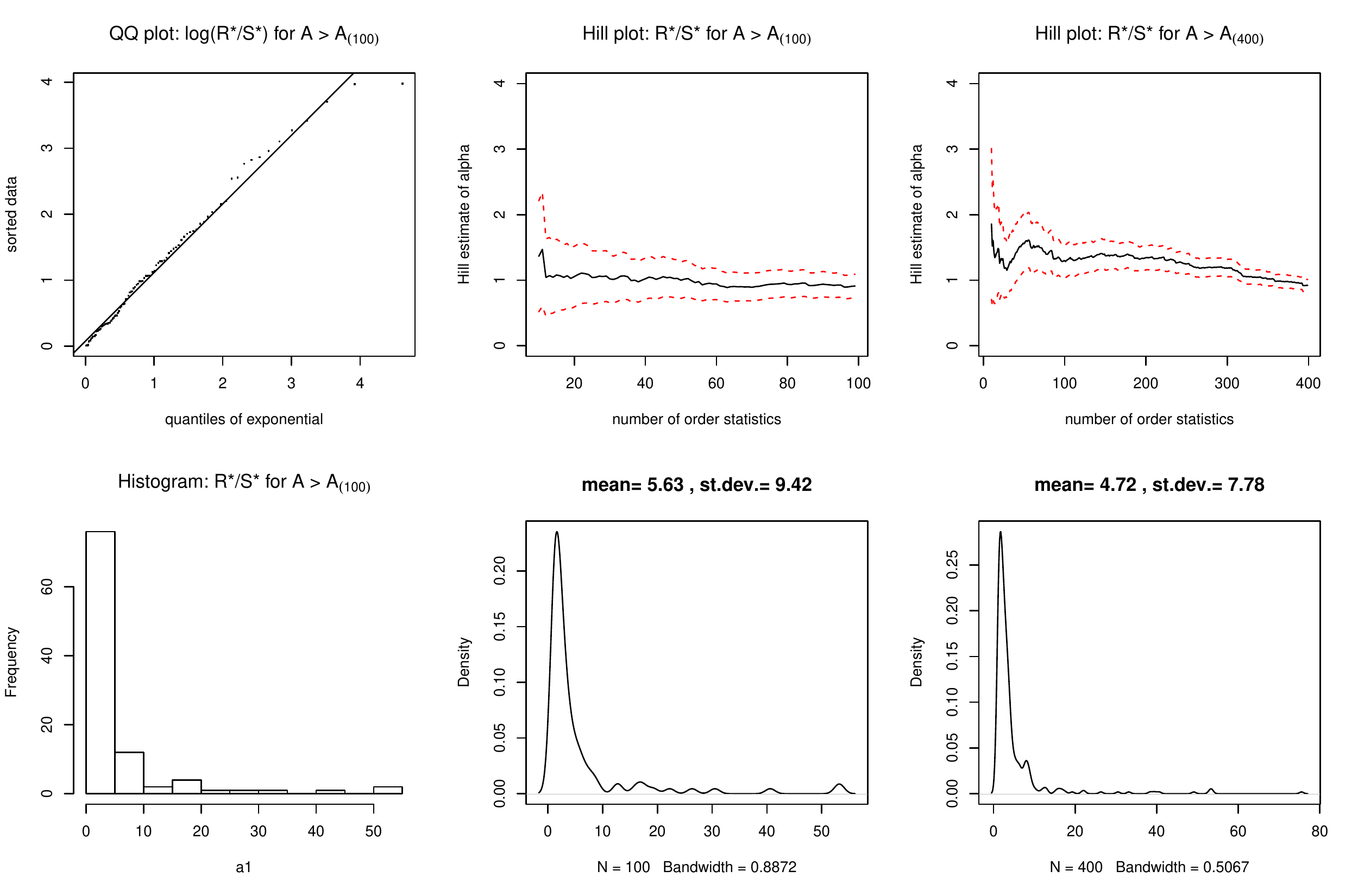}
\end{centering}
\caption{UNC HTTP responses. {\it Top:} QQ plot
of $\log(R^*/S^*)$ when $A_i>A_{(100)}$ and Hill plots of $R^*/S^*$
when $A_i>A_{(100)}$ and $A_i>A_{(400)}$. 
{\it Bottom:} Histogram of $R^*/S^*$ when $A_i>A_{(100)}$ and kernel density estimates of  $R^*/S^*$ for $A_i>A_{(100)}$ and $A_i>A_{(400)}.$ }\label{fig:25Apr_G2}
\end{figure}

So we have accumulated evidence that the data is consistent with an
HRV model on $\E_0$. Now we proceed to provide some estimates on the structure of the hidden angular measure,
which boils down to estimating three things

\begin{enumerate}
\item The proportion $q$ appearing in $\mu_0$ in Proposition \ref{prop:E0}: this can be estimated by 
\[ \hat{q}_k = \frac{1}{k} \sum_{i=1}^{21828} \bone_{\{S_i^*>R_i^*, A_i
  > A_{(k)}\}}, \quad k=2,\ldots, 21,828.\]
where $A_i= \min\{S^*_i,R^*_i\}$ and $A_{(1)} \ge A_{(2)} \ge \ldots$ form
order statistics from $A_i; 1\le i \le 21,828$ as in Figure \ref{fig:25Apr_p}. Looking at the plot for $k$ 
near zero, we can estimate $\hat{p}=0.55$.
\item The distribution of $\theta_1\sim G_1$: see Figure \ref{fig:25Apr_G1}. First we make a standard exponential QQ plot of $\log(S^*_i/R_i^*)$ when $A_{i} > A_{(100)}$.
This acts as a diagnostic for heavy-tails. This plot clearly indicates against heavy-tails as does a Hill plot of  $S^*_i/R_i^*$ when $A_{i} > A_{(100)}$. A histogram and kernel density estimate plot of  $(S^*_i/R_i^*)$ for $A_{i} > A_{(100)}$ points towards a light-tailed distribution.

\item The distribution of $\theta_2\sim G_2$: see Figure \ref{fig:25Apr_G2}. As before, first we make a standard exponential QQ plot of $\log(R^*_i/S_i^*)$ when $A_{i} > A_{(100)}$, and the points nicely hug a straight line which indicates presence of heavy-tails.
The Hill plots of  $R^*_i/S_i^*$ when $A_{i}> A_{(100)}$ and $A_{i}> A_{(400)}$ provide an estimate of the tail index to be between 1 and 1.5. The histograms and kernel density estimates seem to support that the distribution of $G_2$ is heavy-tailed.
\end{enumerate}

\end{exm}

 \section{Conclusion}
In this paper we have discussed different techniques to generate
models which exhibit both regular variation and hidden regular
variation. We have seen  some simulated examples  where we can estimate
the parameters of both MRV and HRV but there are also examples
 where it is difficult to correctly estimate
parameters. We  restricted ourselves to the two dimensional non-negative
orthant here, but clearly some of the generation
techniques can be extended to higher dimensions. Moreover, the
detection techniques for HRV on $\E_0$ using
the CEV model  can also be extended to detect HRV on other types of
cones especially in two dimensions but perhaps even more. Overall this
paper serves as a starting point for methods of generating and
detecting multivariate heavy tailed models having tail dependence
explained
by HRV. 


\bibliography{bibfile}
\end{document}